\documentclass{amsart}

\makeatletter
\@namedef{subjclassname@2020}{%
  \textup{2020} Mathematics Subject Classification}
\makeatother 

\usepackage{amsthm,amssymb,amsfonts,latexsym,mathtools,thmtools}
\usepackage{amsmath}
\usepackage[T1]{fontenc}
\usepackage{tikz-cd} 
\usepackage{enumitem} 
\usepackage{hyperref} 
\hypersetup{
    colorlinks=true,
    linkcolor=blue,
    filecolor=blue,      
    urlcolor=cyan,
    linktocpage=true
}

\newtheorem{theorem}{Theorem}[section]

\theoremstyle{definition}
\newtheorem{definition}[theorem]{Definition}
\newtheorem{proposition}[theorem]{Proposition}
\newtheorem{example}[theorem]{Example}
\newtheorem{remark}[theorem]{Remark}
\newtheorem{corollary}[theorem]{Corollary}

\numberwithin{equation}{section}

\begin{document}

\title[Homogenized skew PBW extensions]{Homogenized skew PBW extensions}



\author{H\'ector Su\'arez}
\address{Universidad Pedag\'ogica y Tecnol\'ogica de Colombia - Sede Tunja}
\curraddr{Campus Universitario}
\email{hector.suarez@uptc.edu.co}
\thanks{}

\author{Armando Reyes}
\address{Universidad Nacional de Colombia - Sede Bogot\'a}
\curraddr{Campus Universitario}
\email{mareyesv@unal.edu.co}
\thanks{}

\author{Y\'esica Su\'arez}
\address{Universidad Pedag\'ogica y Tecnol\'ogica de Colombia - Sede Tunja}
\curraddr{Campus Universitario}
\email{ypsuarezg@unal.edu.co}

\thanks{The authors were supported by the research fund of Faculty of Science, Code HERMES 52464, Universidad Nacional de Colombia - Sede Bogot\'a, Colombia.}



\date{}

\dedicatory{Dedicated to the memory of Professor Vyacheslav Artamonov}

\begin{abstract}

In this paper, we provide a new and more general filtration to the family of noncommutative rings known as skew PBW extensions. We introduce the notion of $\sigma$-filtered skew PBW extension and study some homological properties of these algebras. We show that the homogenization of a $\sigma$-filtered skew PBW extension $A$ over a ring $R$ is a graded skew PBW extension over the homogenization of $R$. Using this fact, we prove that if the homogenization of $R$ is Auslander-regular, then the homogenization of $A$ is a domain Noetherian, Artin-Schelter regular, and $A$ is Noetherian, Zariski and (ungraded) skew Calabi-Yau.

\end{abstract}

\maketitle


\section{Introduction}\label{introduction}

Skew PBW extensions (also known as $\sigma$-PBW extensions) were defined by Ga\-lle\-go and Lezama \cite{GallegoLezama} with the aim of extending the skew polynomial rings (also known as Ore extensions) of injective type introduced by Ore \cite{Ore1933}, and the differential operator rings and PBW extensions defined by Bell and Goodearl \cite{BellGoodearl1988}. Several ring-theoretical properties of skew PBW extensions have been investigated by some authors (\cite{AcostaLezama2015}, \cite{Artamonov1}, \cite{Fajardoetal2020}, \cite{HashemiKhalilAlhevaz2017},  \cite{HashemiKhalilGhadiri2019}, \cite{LouzariReyes2020}, \cite{ReyesSuarez2019Radicals}, \cite{ReyesSuarezYesica}, \cite{TumwesigyeRichterSilvestrov2019} and \cite{Zambrano2020}).

\vspace{0.2cm}

Lezama and the second author \cite{LezamaReyes} defined a filtration for these objects considering the degree of each element of the ring of coefficients as zero (\cite[Theorem 2.2]{LezamaReyes}). Using this filtration, different homological properties have been formulated (e.g. \cite{Fajardoetal2020}, \cite{LezamaGallego2016}, \cite{LezamaGomez2019}, \cite{Suarez}, and \cite{SuarezLezamaReyes2107-1}). Our objective in this paper is to consider the ring of coefficients with a non-zero filtration, and generalize some of the previous results to a more general setting.

\vspace{0.2cm}

The paper is organized as follows. In Section \ref{sect.prelim}, we recall some elementary definitions and properties of ring theory and skew PBW extensions that are needed throughout the paper. Section \ref{sect.fiter skew} contains the definition of filtration on skew PBW extensions over positively filtered algebras, give some remarkable examples and study some properties of these noncommutative rings (Theorem \ref{teo.filtracion gen} and Propositions \ref{prop.filt finita}, \ref{prop.connected filt}, \ref{prop.relat filtr Lezama}, and \ref{prop.libre-filt}). Next, Section \ref{sect.homogeniza} presents properties of $\sigma$-filtered skew PBW extensions over finitely presented algebras (Propositions \ref{prop.R f. pres impl A f. pres}, \ref{prop.filtracion fin pres}, and \ref{prop.rel homogeniz Grad}). We show that the homogenization of a $\sigma$-filtered skew PBW extension over a finitely presented algebra $R$ is a graded skew PBW extension over the homogenization of $R$ (Theorem \ref{teo. homogenization}). In Section \ref{sect.essen Cala}, for $A$ a $\sigma$-filtered skew PBW extension over a ring $R$, we establish properties for $A$, its associated graded ring $G(A)$, and the homogenization $H(A)$ of $A$ (Theorems \ref{teo.propiedades H(A)} and \ref{teo.skew CY}, and Proposition \ref{prop.ex dim 2}). Finally, Section \ref{confutwork} presents some ideas for a possible future work concerning the ideas developed here and topics of interest in computational algebra and noncommutative geometry.
\section{Preliminaries}\label{sect.prelim}
Throughout the paper, the word ring means an associative ring (not
necessarily commutative) with identity. $\mathbb{K}$ denotes a field, all algebras are $\mathbb{K}$-algebras, $\dim(V)$ is the dimension of a $\mathbb{K}$-vector space $V$, all modules are left modules, and the tensor product $\otimes$ means $\otimes_{\mathbb{K}}$. The symbols $\mathbb{N}$, $\mathbb{Z}$ and $\mathbb{C}$ denote the set of natural numbers including zero, the ring of integers, and the field of complex numbers, respectively.

\vspace{0.2cm}

An algebra $R$ is called $\mathbb{Z}$-{\em graded} if there exists a family of subspaces $\{R_p\}_{p\in \mathbb{Z}}$ of $R$ such that $R = \bigoplus_{p\in \mathbb{Z}}R_p$ and $R_pR_q\subseteq R_{p+q}$, for all $p, q\in \mathbb{Z}$. A graded algebra $R$ is called {\em positively graded} (or $\mathbb{N}$-{\em graded}) if $R_p = 0$, for all $p<0$. An $\mathbb{N}$-graded algebra $R$ is called {\em connected} if $R_0 = \mathbb{K}$. A non-zero element $x\in R_p$ is called a {\em homogeneous element} of $R$ of degree $p$. A homogeneous element $z$ of a graded algebra $R$ is said to be {\em regular} if it is neither a left nor a right zero divisor. For $R$ and $S$ two connected graded algebras, if there exists a central element $z\in S_1$ such that $R\cong S/\langle z\rangle$, then $S$ is called a {\em central extension} of $R$. If further $z$ is regular in $S$, then $S$ is called a {\em central regular extension} of $R$. If $R$ is a $\mathbb{Z}$-graded algebra, $R(l) := \bigoplus_{p\in \mathbb{Z}}R(l)_p$, where $R(l)_p=R_{p+l}$, for $l\in \mathbb{Z}$. 

\vspace{0.2cm}

An algebra $R$ is said to be {\em finitely graded} if the following conditions hold:
\begin{itemize}
\item $R$ is $\mathbb{N}$-graded,
\item $R$ is connected,
\item $R$ is \emph{finitely generated} as $\mathbb{K}$-algebra, i.e. there are finite elements $t_1,\dots, t_m\in R$ such that the set
$\{t_{i_1}t_{i_2}\cdots t_{i_p}\mid 1\leq i_j\leq m,\ p\geq 1\} \cup \{1\}$ spans $R$ as a $\mathbb{K}$-space.
\end{itemize}
A {\em filtration} $\mathcal{F}$ on an algebra $R$ is a collection
of vector spaces $\{\mathcal{F}_p(R)\}_{p\in \mathbb{Z}}$ such that
$\mathcal{F}_p(R) \subseteq \mathcal{F}_{p+1}(R),$\
$\mathcal{F}_p(R)\cdot\mathcal{F}_q(R)\subseteq
\mathcal{F}_{p+q}(R)$, for every $p, q \in \mathbb{Z}$, and $\bigcup_{p\in \mathbb{Z}}\mathcal{F}_p(R)=R$. The filtration $\mathcal{F}$ is said to be \emph{finite} if each $\mathcal{F}_p(R)$ is a finite dimensional subspace. The filtration is \emph{positive} if $\mathcal{F}_{-1}(R)= 0$. In this case, we say that $R$ is \emph{positively} filtered ($\mathbb{N}$-filtered). If $0\neq r\in\mathcal{F}_p(R)\setminus \mathcal{F}_{p-1}(R)$, then $r$ is said to have \emph{degree} $p$, and write $\deg(r)=p$. A positive filtration is said to be \emph{connected} if $\mathcal{F}_0(R) = \mathbb{K}$; in this case, we say that $R$ is \emph{connected filtered}. The associated graded algebra of $R$ is given by $G_{\mathcal{F}}(R):= \bigoplus_{p\geq 0}\mathcal{F}_p(R)/\mathcal{F}_{p-1}(R)$. Notice that $G_{\mathcal{F}}(R)$ is connected if the filtration $\mathcal{F}$ is connected. We simply write $G(R)$ if no confusion arises.

\vspace{0.2cm}

The associated \emph{Rees algebra} is defined as ${\rm Rees}_{\mathcal{F}}(R): = \bigoplus_{p\geq 0}\mathcal{F}_p(R)z^p$. The filtration $\{F_p(R)\}_{p\in \mathbb{Z}}$ is \emph{left} ({\em right}) {\em Zariskian} and $R$ is called a \emph{left} ({\em right}) {\em Zariski ring} if $F_{-1}(R)\subseteq {\rm Rad}(F_0(R))$ (where ${\rm Rad}(F_0(R))$ is the Jacobson radical of $F_0(R)$), and the associated Rees ring ${\rm Rees}_{\mathcal{F}}(R)$ is left (right) Noetherian. Of course, if $R$ is graded, then $R=G(R)$. In this case, we write $R_p$ for the vector space spanned by homogeneous elements of degree $p$. If $R$ is a filtered algebra with filtration $\{\mathcal{F}_p(R)\}_{p\in \mathbb{Z}}$ and $M$ is an $R$-module, then we say that $M$ is \emph{filtered} if there exists a family $\{\mathcal{F}_p(M)\}_{p\in \mathbb{Z}}$ of subspaces of $M$ such that $\mathcal{F}_p(M) \subseteq \mathcal{F}_{p+1}(M)$,\ $\mathcal{F}_p(R)\cdot\mathcal{F}_q(M)\subseteq \mathcal{F}_{p+q}(M)$, and $\bigcup_{p\in \mathbb{Z}}\mathcal{F}_p(M)=M$. If $m\in M_p\setminus M_{p-1}$, then $m$ is said to have \emph{degree} $p$. For further details about filtered and Rees rings, see \cite{Huishivan1996}.

\vspace{0.2cm}

For $R$ a connected graded algebra, its \emph{global homological dimension} ${\rm gld}(R)$ is the projective dimension of the trivial $R$-module $\mathbb{K} = R/R_+$, where $R_+$ is the augmentation ideal generated by all degree one elements. If $V$ is a generating set for $R$ and $V^n$ is the set of elements of degree $n$, then the \emph{Gelfand-Kirillov dimension} of $R$ is defined as ${\rm GKdim}(R):= \overline{\rm lim} _{n\to\infty}{\rm log}_n(\dim V^n)$. The algebra $R$ is said to be Artin-Schelter \emph{Gorenstein} if ${\rm Ext}^i_R({\mathbb{K}_R},R)\cong\delta_{i,d}{\mathbb{K}}$, where $\delta_{i,d}$ is the Kronecker delta and $d = {\rm gld}(R)$.

\vspace{0.2cm}

The {\em free associative algebra} $L$ in $m$ generators $t_1,\dots, t_m$, denoted by $L:=\mathbb{K}\langle t_1,\dots, t_m\rangle$, is the ring whose underlying $\mathbb{K}$-vector space  is the set of all words in the indeterminates $t_i$, that is, expressions of the form $t_{i_1}t_{i_2}\dots\ t_{i_p}$, for some $p\geq 1$, where $1\leq i_j \leq m$, for all $j$. The \emph{degree} ($\deg$) of a word $t_{i_1}t_{i_2}\dots t_{i_p}$ is $p$, and the degree of an element $f\in L$ is the maximum of the degrees of the words in $f$. We include among the words a symbol 1, which we think of as the empty word, and which has degree 0. The product of two words is concatenation, and this operation is extended linearly to define an associative product on all elements. Notice that $L$ is positively graded with graduation given by $L:=\bigoplus_{p \geq 0}L_p$, where $L_0= \mathbb{K}$ and $L_p$ spanned by all words of degree $p$ in the alphabet $\{t_1, \dots, t_m\}$, for $p>0$. $L$ is connected and therefore augmented, where the augmentation of $L$ is given by the natural projection $\varepsilon: \mathbb{K}\langle t_1,\dots, t_n\rangle\to L_0= \mathbb{K}$ and the augmentation ideal is given by $L_+:=\bigoplus_{p > 0}L_p$. $L$ is connected filtered with the standard filtration $\{\mathcal{F}_q(L)\}_{q\in \mathbb{N}}$, where $\mathcal{F}_q(L)=\bigoplus_{p\leq q} L_p$.

\vspace{0.2cm}

An algebra $R$ is \emph{finitely presented} if it is a quotient of the form $\mathbb{K}\langle t_1, \dots, t_m\rangle/I$ where $I$ is a finitely generated two-sided ideal of $\mathbb{K}\langle t_1, \dotsc, t_m\rangle$, say $I = \langle r_1,\dots, r_s\rangle$. $\mathbb{K}\langle t_1, \dots, t_m\rangle/I$ is said to be a \emph{presentation} of $R$ with generators $t_1, \dots, t_m$ and relations $r_1,\dots, r_s$. Throughout the paper, we assume that $\{r_1,\dots, r_s\}$ is a minimal set of relations for $R$, the generators $t_i$ all have degree 1, and none of the relations $r_i$ are linear. Notice that if the relations $r_1,\dots, r_s$ are all homogeneous, then $R$ is called a {\em connected graded algebra}. Now, by a \emph{deformation} of a connected graded algebra $R$ we mean an algebra
\begin{equation}\label{eq.deform}
U = \mathbb{K}\langle t_1, \dots, t_m\rangle/\langle r_1+l_1,\dots, r_s+l_s\rangle,
\end{equation}
where $l_1,\dots, l_s$ are (not necessarily homogenous) elements of $\mathbb{K}\langle t_1, \dots, t_m\rangle$ such that $\deg(l_i) < \deg(r_i)$, for all $i$. There is a standard filtration on $U$ induced by the standard filtration on $\mathbb{K}\langle t_1, \dots, t_m\rangle$. If $g = \sum_{k=0}^pg_k\in \mathbb{K}\langle t_1, \dots, t_m\rangle,$ where each nonzero $g_k$ is a homogeneous polynomial of degree $k$, and $\deg(g_1) < \deg(g_2) <\cdots < \deg(g_p)$, then $g_p$ is said to be the \emph{leading homogeneous polynomial} of $g$, which is denoted by ${\rm lh}(g)$. The \emph{homogenization} $\widehat{g}$ of $g$ is given by $\widehat{g}=\sum_{k=0}^pg_kz^{p-k}$, where $z$ is a new central indeterminate. Let $R = \mathbb{K}\langle t_1, \dots, t_m\rangle/\langle f_1,\dots, f_s\rangle$ be a finitely presented algebra. Since $R$ is not necessarily graded, if we homogenize every polynomial $f_i \in R$, we obtain a graded algebra known as the {\em homogenization of} $R$.

\vspace{0.2cm}

In the setting of noncommutative rings having PBW bases, Cassidy and Shelton \cite[Theorem 1.3]{Cassidy2007} proved that a deformation $U$ of the graded algebra $R$ is a PBW deformation if and only if the homogenization of $U$ is a regular central extension. Other properties of central extensions and homogenization have been used by several authors to study certain classes of algebras (e.g. \cite{Cassidy2007}, \cite{Chirvasitu2018}, \cite{Gaddis2016}, \cite{Shen2016}, and \cite{Wu2013}). 

\vspace{0.2cm}

Of interest for us in this paper, we recall the following definition.
\begin{definition}(\cite[Definition 2.1]{Gaddis2016}). 
Let $U = \mathbb{K}\langle t_1, \dots, t_m\rangle/\langle f_1,\dots, f_s\rangle$ be an algebra, where $f_1,\dots, f_s\in L = \mathbb{K}\langle t_1, \dots, t_m\rangle$. The graded algebra $H(U) = L[z]/\langle\widehat{f}_1, \dots, \widehat{f}_s\rangle$ is called the \emph{homogenization} of $U$.
\end{definition}
In other words, the homogenization $H(R)$ of $R = L/\langle f_1,\dots, f_s\rangle$ is the algebra with $n + 1$ generators $t_1, \dots, t_n$, and $z$, subject to the homogenized relations $\widehat{f}_k$ as well as the additional relations $zt_i-t_iz$, for $1\leq i\leq n$.

\vspace{0.2cm}

Notice that if $U := \mathbb{K}\langle t_1, \dots, t_m\rangle/\langle f_1,\dots, f_s\rangle$ is an algebra and we consider
\begin{equation*}
R:= \mathbb{K}\langle t_1, \dots, t_m\rangle/\langle {\rm lh}(f_1),\dots, {\rm lh}(f_s)\rangle,
\end{equation*}
then we have a natural graded surjective homomorphism $\phi: R\to G(U)$.
When $\phi$ is an isomorphism, we say that $U$ is a {\em Poincar\'e-Birkhoff-Witt} ({\em PBW}) \emph{deformation} of $G(U)$ \cite[Definition 2.6]{Gaddis2016}. If $R$ is a connected graded algebra, a deformation $U$ of $R$ as in (\ref{eq.deform}) is said to be a {\em PBW deformation} if $G(U)$ is isomorphic to $R$.
\begin{definition}(\cite{Artin1987}). 
A connected graded algebra $R$ is said to be \emph{Artin-Schelter regular} of dimension $d$ if:
\begin{enumerate}
\item[\rm (1)] $R$ has finite global dimension $d$;
\item[\rm (2)] $R$ has finite Gelfand-Kirillov dimension;
\item[\rm (3)] ${\rm Ext}_R^i(\mathbb{K}, R) = 0$ if $i \neq d$, and ${\rm Ext}^d_R(\mathbb{K}, R)\cong \mathbb{K}$.
\end{enumerate}
\end{definition}
Now, we recall the definition of skew PBW extension and some of its properties.
\begin{definition}(\cite[Definition 1]{GallegoLezama}).\label{def.skewpbwextensions}
Let $R$ and $A$ be rings. We say that $A$ is a \textit{skew PBW extension over} $R$ (also called a $\sigma$-{\em PBW extension over} $R$) if the following conditions hold:
\begin{enumerate}
\item[\rm (1)]$R$ is a subring of $A$ sharing the same identity element.
\item[\rm (2)]There exist finitely many elements $x_1,\dots ,x_n\in A$ such that $A$ is a  free $R$-module, with basis the basic elements ${\rm Mon}(A):= \{x^{\alpha}=x_1^{\alpha_1}\cdots x_n^{\alpha_n}\mid \alpha=(\alpha_1,\dots ,\alpha_n)\in \mathbb{N}^n\}$.
\item[\rm (3)]For each $1\leq i\leq n$ and any $r\in R\ \backslash\ \{0\}$, there exists an element $c_{i,r}\in R\ \backslash\ \{0\}$ such that $x_ir-c_{i,r}x_i\in R$.
\item[\rm (4)]For any elements $1\leq i,j\leq n$, there exists $d_{i,j}\in R\ \backslash\ \{0\}$ such that
\begin{equation}\label{sigmadefinicion2}
x_jx_i-d_{i,j}x_ix_j\in R+Rx_1+\cdots +Rx_n.
\end{equation}
Under these conditions, we write $A:=\sigma(R)\langle x_1,\dots,x_n\rangle$.
\end{enumerate}
\end{definition}
For $X = x^{\alpha}=x_1^{\alpha_1}\cdots x_n^{\alpha_n}\in {\rm Mon}(A)$, $\deg(X):= {\alpha_1}+\cdots + {\alpha_n}$. The relationship between skew polynomial rings and skew PBW extensions is presented in the following proposition.
\begin{proposition}(\cite[Proposition 3]{GallegoLezama}). \label{sigmadefinition}
Let $A$ be a skew PBW extension over $R$. For each $1\leq i\leq n$, there exist an injective endomorphism $\sigma_i:R\rightarrow R$ and a $\sigma_i$-derivation $\delta_i:R\rightarrow R$ such that $x_ir=\sigma_i(r)x_i+\delta_i(r)$, where $r\in R$.
\end{proposition}
From now on, $\sigma_i$ and $\delta_i$ are the injective endomorphisms and the $\sigma_i$-derivations as in Proposition \ref{sigmadefinition}, respectively. 

\vspace{0.2cm}

A skew PBW extension $A$ is called \textit{bijective} if $\sigma_i$ is bijective and $d_{i,j}$ is invertible, for any $1\leq i<j\leq n$. $A$ is called \textit{quasi-commutative} if the conditions {\rm(}3{\rm)} and {\rm(}4{\rm)} in Definition \ref{def.skewpbwextensions} are replaced by the following:
\begin{enumerate}
\item[\rm (3')] for each $1\leq i\leq n$ and all $r\in R\ \backslash\ \{0\}$, there exists $c_{i,r}\in R\ \backslash\ \{0\}$ such that $x_ir=c_{i,r}x_i$;
\item[\rm (4')]for any $1\leq i,j\leq n$, there exists $d_{i,j}\in R\ \backslash\ \{0\}$ such that $x_jx_i=d_{i,j}x_ix_j$.
\end{enumerate}
Examples of bijective and quasi-commutative skew PBW extensions, and some of their properties can be found in \cite{Fajardoetal2020}, \cite{HashemiKhalilAlhevaz2019},  \cite{ReyesSuarez2019-1}, \cite{ReyesSuarez2019-2}, and \cite{ReyesSuarez2019Radicals} .

\vspace{0.2cm}

Let $I\subseteq \sum_{n\geq 2} L_{n}$ be a finitely generated homogeneous ideal of $\mathbb{K}\langle t_1,\dots, t_m\rangle$ and let $R = \mathbb{K}\langle t_1,\dots, t_m\rangle/I$ be a connected
graded algebra generated in degree 1. Suppose that $\sigma : R \to R$ is a graded algebra automorphism and $\delta : R(-1) \to R$ is a graded $\sigma$-derivation (i.e. a degree +1 graded
$\sigma$-derivation $\delta$ of $R$). Let  $B := R[x; \sigma,\delta]$ be the associated \emph{graded Ore extension} of $R$, that is, $B = \bigoplus_{p\geq 0} Rx^p$ as an $R$-module, and for $r\in R$, $xr = \sigma(r)x + \delta(r)$. If we consider $x$ to have degree 1 in $B$, then under this grading $B$ is a connected graded algebra generated in degree 1 (see \cite{Cassidy2008} and \cite{Phan} for more details).
\begin{proposition}(\cite[Proposition 2.7]{Suarez}). \label{prop.grad A}
Let $R=\bigoplus_{m\geq 0}R_m$ be an $\mathbb{N}$-graded algebra, and let $A=\sigma(R)\langle x_1,\dots, x_n\rangle$ be a bijective skew PBW extension over $R$ satisfying the following two conditions:
\begin{enumerate}
\item[\rm (1)] $\sigma_i$ is a graded ring homomorphism and $\delta_i : R(-1) \to R$ is a graded $\sigma_i$-derivation, for all $1\leq i  \leq n$.
\item[\rm (2)]  $x_jx_i-d_{i,j}x_ix_j\in R_2+R_1x_1 +\cdots + R_1x_n$, as in {\rm (\ref{sigmadefinicion2})} and $d_{i,j}\in R_0$.
\end{enumerate}
For $p\geq 0$, if $A_p$ is the $\mathbb{K}$-space generated by the set
\[\Bigl\{r_tx^{\alpha} \mid t+\deg(x^{\alpha})= p,\  r_t\in R_t \text{  and } x^{\alpha}\in {\rm Mon}(A)\Bigr\},
\]
then $A$ is an $\mathbb{N}$-graded algebra given by $A = \bigoplus_{p\geq 0} A_p$.
\end{proposition}
Proposition \ref{prop.grad A} motivates the following definition.
\begin{definition}(\cite[Definition 2.6]{Suarez}). \label{def. graded skew PBW ext} Let  $A=\sigma(R)\langle x_1,\dots, x_n\rangle$ be a bijective skew PBW extension over an $\mathbb{N}$-graded algebra $R=\bigoplus_{m\geq 0}R_m$. $A$ is said to be a \emph{graded skew PBW extension over} $R$ if it satisfies the conditions (1) and (2) established in Proposition \ref{prop.grad A}.
\end{definition}
The class of graded iterated Ore extensions of injective type is strictly contained in the class of graded skew PBW extensions. For example, homogenized enveloping algebras and diffusion algebras are graded skew PBW extensions over a field but these are not iterated Ore extensions of the field. Details and examples of graded skew PBW extensions can be found in \cite{GomezSuarez2019},  \cite{Suarez} and \cite{SuarezCaceresReyes2021}.
\section{\texorpdfstring{$\sigma -$}\ filtered skew PBW extensions}\label{sect.fiter skew}
If $R$ is an arbitrary algebra, then it is clear that $R$ is a filtered algebra with filtration given by $\mathcal{F}_p(R)=R$, for all $p\in \mathbb{Z}$. In this case, we say that  $R$ has the {\em trivial filtration}. $R$ has the {\em  trivial positive filtration} if $\mathcal{F}_p(R)=R$, for all $p\geq 0$ and $\mathcal{F}_{-1}(R)=0$. If $l\geq 0$, then $R$ is connected filtered with filtration given by
\begin{equation*}\label{eq.connec triv}
\mathcal{F}_p(R)= \left\{
  \begin{array}{ll}
    0, & \text{ if } p=-1;\\
    \mathbb{K}, & \text{ if } 0\leq p\leq l; \\
    R, & \text{ if } p>l.
  \end{array}
\right.
\end{equation*}
In this case, we say that  $R$ is an \emph{$l$-trivial connected filtered algebra}. If $l=0$, then we say that  $R$ is a \emph{trivial connected filtered algebra}. Throughout the paper, we assume that $\mathbb{K}$ has trivial connected filtration.
\begin{remark}\label{rem.grade of f} 
Let $A=\sigma(R)\langle x_1,\dots, x_n\rangle$ be a skew PBW extension.
\begin{enumerate}
\item[\rm (1)] From (\ref{sigmadefinicion2}), we have that
\begin{equation}\label{eq.rep variables x}
x_jx_i=d_{i,j}x_ix_j+r_{0_{j,i}} + r_{1_{j,i}}x_{1} + \cdots + r_{n_{j,i}}x_{n},
\end{equation}
where $d_{i,j},r_{0_{j,i}},  r_{1_{j,i}}, \dots, r_{n_{j,i}}\in R$, for $1\leq i,j\leq n$.
\item[\rm (2)] (\cite[Remark 2]{GallegoLezama}). Every element $f\in A\ \backslash\ \{0\}$ has a unique representation as
\begin{equation}\label{eq.rep unica of f}
f=c_1X_1+\cdots+c_dX_d,\ \text{ with }\ c_i\in R\ \backslash\ \{0\} \text { and } X_i\in {\rm Mon}(A), \text{ for } 1\leq i\leq d.
\end{equation}
\end{enumerate}
\end{remark}
\begin{definition}\label{def.total degree} 
Let $A=\sigma(R)\langle x_1,\dots, x_n\rangle$ be a skew PBW extension over a positively filtered algebra $R=\bigcup_{p\in
\mathbb{N}}\mathcal{F}_p(R)$.
\begin{enumerate}
\item[\rm (1)] For $X = x^{\alpha}=x_1^{\alpha_1}\cdots x_n^{\alpha_n}\in {\rm Mon}(A)$ and $c\in R\ \backslash\ \{0\}$, ${\rm tdeg}(cX):= \deg(c)+\deg(X)$.
\item[\rm (2)] Let $f=c_1X_1+\cdots+c_dX_d\in A\ \backslash\ \{0\}$, ${\rm tdeg}(f):= \max\{{\rm tdeg}(c_iX_i)\}_{i=1}^d$.
\item[\rm (3)] Let $\sigma: R\to R$ be an endomorphism of algebras. If $\sigma(\mathcal{F}_p(R))\subseteq \mathcal{F}_p(R)$, then we say
that $\sigma$ is a \emph{filtered endomorphism}.
\item[\rm (4)] Let $\delta: R\to R$ be a $\sigma$-derivation. If $\delta(\mathcal{F}_p(R))\subseteq \mathcal{F}_{p+1}(R)$, we say that $\delta$ is a \emph{filtered} $\sigma$-derivation, and if
$\delta(\mathcal{F}_p(R))\subseteq \mathcal{F}_{p+m}(R)$, then we say that $\delta$ is an $m$-filtered $\sigma$-derivation, for $m>1$.
\item[\rm (5)] We say that $A$ \emph{preserves} tdeg if for each $x_{j}x_{i}$ as in (\ref{eq.rep variables x}),
\begin{equation*}
{\rm tdeg}(x_jx_i) = {\rm tdeg}(c_{i,j}x_{i}x_{j} + r_{0_{j,i}} + r_{1_{j,i}}x_{1} + \cdots + r_{n_{j,i}}x_{n})=2.
\end{equation*}
\end{enumerate}
\end{definition}
The following theorem, one of the most important results of the paper, provides a special filtration to the skew PBW extensions.
\begin{theorem}\label{teo.filtracion gen}
If $A=\sigma(R)\langle x_1,\dots, x_n\rangle$ is a skew PBW extension over a positively filtered algebra $R$ such that the following conditions hold:
\begin{enumerate}
\item[\rm (1)] $\sigma_i$ and $\delta_i$ are filtered, for $1\leq i\leq n$;
\item[\rm (2)]  $A$ preserves {\rm tdeg},
\end{enumerate}
then $\{\mathcal{F}_p(A)\}_{p\in \mathbb{N}}$ is a filtration on $A$, where
\begin{equation}\label{eq.filtr A}
\mathcal{F}_p(A):=\{f\in A\mid {\rm tdeg}(f)\le p\}\cup \{0\}.
\end{equation}
Moreover, $A$ is a filtered $R$-module with the same filtration.
\end{theorem}
\begin{proof}
Notice that for each $f\in A\ \backslash\ \{0\}$, ${\rm tdeg}(f)\geq 0$. By definition, $0\in\mathcal{ F}_p(A)$. Let $f,g\in \mathcal{F}_p(A)$ with $f\neq 0$ and $g\neq 0$. By Remark \ref{rem.grade of f} (2), $f$ and $g$ have a unique representation as $f=c_1X_1+\cdots+c_dX_d$ and $g=r_1Y_1+\cdots+r_eY_e$,  with $c_i, r_j\in R\ \backslash\ \{0\}$ and $X_i, Y_j\in {\rm Mon}(A)$ for
$1\leq i\leq d$ and $1\leq j\leq e$. In this way, ${\rm tdeg}(c_iX_i)\leq p$ and ${\rm tdeg}(r_jY_j)\leq p$, for $1\leq
i\leq d$, $1\leq j\leq e$. Thus ${\rm tdeg}(f+g) = {\rm tdeg}(c_1X_1+\cdots+c_dX_d+r_1Y_1+\cdots+r_eY_e)=\max\{{\rm
tdeg}(c_iX_i), {\rm tdeg}(e_jY_j)\mid 1\leq i\leq d, 1\leq j\leq e\}\leq p$, and so $f+g\in \mathcal{F}_p(A)$. Now, if $k\in
\mathbb{K}$, then ${\rm tdeg}(kf) = {\rm tdeg}((kc_1)X_1+\cdots+(kc_d)X_d)\leq p$, whence $\mathcal{F}_p(A)$ is a subspace of $A$, for each $p\in \mathbb{N}$. It is clear that $\bigcup_{p\in \mathbb{N}}\mathcal{F}_p(A)=A$. If $0\neq f=c_1X_1+\cdots+c_dX_d\in\mathcal{F}_p(A)$, then ${\rm tdeg}(c_iX_i)\leq p<p+1$, for $1\leq i\leq d$, and $\mathcal{F}_p(A)\subseteq \mathcal{F}_{p+1}(A)$. Let $h\in \mathcal{F}_p(A)\cdot\mathcal{F}_q(A)$. Without loss of generality we assume that $h= h_ph_q$ with $h_p\in\mathcal{F}_p(A)$ and $h_q\in\mathcal{F}_q(A)$. Let $h_p = a_1X_1 +\cdots + a_mX_m$, $h_q = b_1Y_1 +\cdots + b_tY_t$. Then ${\rm tdeg}(a_iX_i)\leq p$ and ${\rm tdeg}(b_jY_j)\leq q$, for $1\leq i\leq m$, $1\leq j\leq t$.
Hence
\begin{equation}\label{eq.prod}
h= (a_1X_1 +\cdots + a_mX_m)(b_1Y_1 +\cdots + b_tY_t) =
\sum_{k=1}^{m+t}(\sum_{i+j=k}a_iX_ib_jY_j),
\end{equation}
and so ${\rm tdeg}(h)=\max\{{\rm tdeg}(a_iX_ib_jY_j)\mid 1\leq i\leq m, 1\leq j\leq t\}$, but obtaining the unique representation of $a_iX_ib_jY_j$ as in (\ref{eq.rep unica of f}) once the commutation rules have been made taking into account (3) and (4) in the Definition \ref{def.skewpbwextensions}.

\vspace{0.2cm}

Let $X_i=x_1^{\alpha_1}\cdots x_n^{\alpha_n}$, $Y_j=x_1^{\beta_1}\cdots x_n^{\beta_n}$, with $\alpha_i,\beta_i\in\mathbb{N}$. By \cite[Remark 2.7]{ReyesSuarez2017-3} we have that {\scriptsize{
\begin{align}\label{eq.sumas}
\begin{split}
a_i&X_ib_jY_j = a_i(x_1^{\alpha_1}\cdots x_n^{\alpha_n}b_j)x_1^{\beta_1}\cdots x_n^{\beta_n}\\
= &  \ a_ix_1^{\alpha_1}\dotsb
x_{n-1}^{\alpha_{n-1}}\biggl(\sum_{j=1}^{\alpha_n}x_n^{\alpha_{n}-j}\delta_n(\sigma_n^{j-1}(b_j))x_n^{j-1}\biggr)x_1^{\beta_1}\cdots
x_n^{\beta_n}\\
+ &\ a_ix_1^{\alpha_1}\dotsb x_{n-2}^{\alpha_{n-2}}\biggl(\sum_{j=1}^{\alpha_{n-1}}x_{n-1}^{\alpha_{n-1}-j}\delta_{n-1}(\sigma_{n-1}^{j-1}(\sigma_n^{\alpha_n}(b_j)))x_{n-1}^{j-1}\biggr)x_n^{\alpha_n}x_1^{\beta_1}\cdots x_n^{\beta_n}\\
+ &\ a_ix_1^{\alpha_1}\dotsb x_{n-3}^{\alpha_{n-3}}\biggl(\sum_{j=1}^{\alpha_{n-2}} x_{n-2}^{\alpha_{n-2}-j}\delta_{n-2}(\sigma_{n-2}^{j-1}(\sigma_{n-1}^{\alpha_{n-1}}(\sigma_n^{\alpha_n}(b_j))))x_{n-2}^{j-1}\biggr)x_{n-1}^{\alpha_{n-1}}x_n^{\alpha_n}x_1^{\beta_1}\cdots x_n^{\beta_n}\  +\dotsb \\
+ & \ a_ix_1^{\alpha_1}\biggl( \sum_{j=1}^{\alpha_2}x_2^{\alpha_2-j}\delta_2(\sigma_2^{j-1}(\sigma_3^{\alpha_3}(\sigma_4^{\alpha_4}(\dotsb (\sigma_n^{\alpha_n}(b_j))))))x_2^{j-1}\biggr)x_3^{\alpha_3}x_4^{\alpha_4}\dotsb x_{n-1}^{\alpha_{n-1}}x_n^{\alpha_n}x_1^{\beta_1}\cdots x_n^{\beta_n} \\
+ &\ a_i\sigma_1^{\alpha_1}(\sigma_2^{\alpha_2}(\dotsb
(\sigma_n^{\alpha_n}(b_j))))x_1^{\alpha_1}\dotsb
x_n^{\alpha_n}x_1^{\beta_1}\cdots x_n^{\beta_n}, \ \ \ \ \ \ \ \ \ \
\sigma_j^{0}:={\rm id}_R\ \ {\rm for}\ \ 1\le j\le n. \end{split}
\end{align}}}
Notice that as $\sigma_i$ and $\delta_i$ are filtered, then
$\sigma_i^k(\mathcal{ F}_p(R))\subseteq \sigma_i^{k-1}(\mathcal{
F}_p(R))\subseteq\cdots\subseteq \sigma_i(\mathcal{ F}_p(R))$ and $\delta_i^k(\mathcal{ F}_p(R))\subseteq \mathcal{ F}_{p+k}(R))$, and thus $\sigma_i^k$ is filtered  and $\delta_i^k$ is $k$-filtered. Furthermore, as $A$ preserves ${\rm tdeg}$, then for each of the summands in (\ref{eq.sumas}), ${\rm tdeg}\leq p+q$ (once the commutation rules have been made taking into account (3) and (4) in the Definition \ref{def.skewpbwextensions}). In this way, ${\rm tdeg}(h)\leq p+q$, and so $h\in \mathcal{F}_{p+q}$, whence $\{\mathcal{F}_p(A)\}_{p\in \mathbb{N}}$ is a filtration on $A$.

\vspace{0.2cm}

On the other hand, let $g\in \mathcal{F}_p(R)\mathcal{F}_q(A)$. Then $g=rf$, for some $r\in \mathcal{F}_p(R)$ and $f=r_1X_1+\cdots + r_tX_t\in \mathcal{F}_q(A)$. So, ${\rm tdeg}(r_iX_i)=
\deg(r_i)+\deg(X_i)\leq q$, for $1\leq i\leq t$. Thus ${\rm tdeg}(r(r_iX_i)) = {\rm  tdeg}((rr_i)X_i))=\deg(rr_i)+\deg(X_i)\leq p+\deg(r_i)+\deg(X_i)\leq p+q$, which implies that ${\rm tdeg}(rf)\leq p+q$, and therefore $rf=g\in \mathcal{F}_{p+q}(A)$. This shows that $A$ is a filtered $R$-module.
\end{proof}
Theorem \ref{teo.filtracion gen} suggests the following definition.
\begin{definition}\label{def.filt skew} 
Let $A=\sigma(R)\langle x_1,\dots, x_n\rangle$ be a skew PBW extension over a positively filtered algebra $R$. We say that $A$ is a $\sigma$-{\em filtered skew PBW extension over} $R$ if $A$ satisfies the conditions (1) and (2) in Theorem \ref{teo.filtracion gen}.
\end{definition}
In this case, it is understood that $A$ has the filtration $\{\mathcal{F}_p(A)\}_{p\in \mathbb{N}}$, where $\mathcal{F}_p(A)$ is as in (\ref{eq.filtr A}).
\begin{proposition}\label{prop.filt finita}
Let $A=\sigma(R)\langle x_1,\dots, x_n\rangle$ be a $\sigma$-filtered skew PBW extension over $R$.
\begin{enumerate}
\item[\rm (1)] If $\{\mathcal{F}_p(R)\}_{p\in \mathbb{N}}$ is a positive filtration on $R$, then $\mathcal{F}_p(R)$ is a subspace of $\mathcal{F}_p(A)$.
\item[\rm (2)] If the filtration on $R$ is finite, then the filtration of $A$ is finite.
\end{enumerate}
\end{proposition}
\begin{proof}
Let $\{\mathcal{F}_p(R)\}_{p\in \mathbb{N}}$ be a connected filtration on $R$ and $\{\mathcal{F}_p(A)\}_{p\in \mathbb{N}}$ the filtration on $A$.
\begin{enumerate}
\item[\rm (1)] If $0\neq r\in \mathcal{F}_p(R)$, then $\deg(r)\leq p$. By Definition \ref{def.skewpbwextensions}(1), we have $R\subseteq A$,
and $r=rx_1^0\cdots x_n^0$ is the unique representation of $r$. This means that ${\rm tdeg}(r)=\deg(r)\leq p$, and so $r\in
\mathcal{F}_p(A)$.
\item[\rm (2)] Let $\mathcal{B}^{(R)}_k$ be a finite basis for $\mathcal{F}_k(R)$. By Definition \ref{def.skewpbwextensions}(2), we have that $A$ is a left free $R$-module with basis ${\rm Mon}(A) = \{x^{\alpha}=x_1^{\alpha_1}\cdots x_n^{\alpha_n}\mid \alpha=(\alpha_1,\dots ,\alpha_n)\in \mathbb{N}^n\}$, whence $\mathcal{B}_p= (\bigcup_{k=0}^p\mathcal{B}^{(R)}_k)\bigcup \{X\in {\rm Mon}(A)\mid \deg(X)\leq p\}$ is a finite basis for $\mathcal{F}_p(A)$.
\end{enumerate}
\end{proof}
\begin{proposition}\label{prop.connected filt}
If $A=\sigma(R)\langle x_1,\dots, x_n\rangle$ is a $\sigma$-filtered skew PBW extension over $R$, then $R$ is connected filtered if and only if $A$ is  connected filtered.
\end{proposition}
\begin{proof}
Let $A=\sigma(R)\langle x_1,\dots, x_n\rangle= \bigcup_{p\in \mathbb{N}}\mathcal{F}_p(A)$ be a $\sigma$-filtered skew PBW extension over a connected filtered algebra $R$. If $0\neq h\in \mathcal{F}_0(A)$, then $h$ has a unique representation as $h=a_1X_1+\cdots+ a_tX_t$ (Remark \ref{rem.grade of f} (2)), whence ${\rm tdeg}(a_iX_i)=\deg(a_i)+\deg(X_i)=0$, for $1\leq i\leq t$. Thus, $\deg(a_i)=0=\deg(X_i)$, $X_i=1$, for $1\leq i\leq t$, and so $h\in R$ with $\deg(h)=0$, that is, $h\in \mathcal{F}_0(R)=\mathbb{K}$ since $R$ is connected filtered. This proves that $\mathcal{F}_0(A)=\mathbb{K}$, i.e. $A$ is connected filtered.

\vspace{0.2cm}

For the converse, let $0\neq r\in \mathcal{F}_0(R)$. By Proposition \ref{prop.filt finita}(1), $\mathcal{F}_0(R)\subseteq \mathcal{F}_0(A)=\mathbb{K}$, whence $r\in \mathbb{K}$, that is, $\mathcal{F}_0(R)=\mathbb{K}$.
\end{proof}
Lezama and Reyes \cite[Theorem 2.2]{LezamaReyes} defined a filtration $\mathcal{F}'$ for a skew PBW extension $A=\sigma(R)\langle x_1,\dots, x_n\rangle$ in the following way:
\begin{equation}\label{eq1.3.1a}
\mathcal{F}'_m(A):=\begin{cases} R, & {\rm if}\ \ m=0\\ \{f\in A\mid {\rm \deg}(f)\le m\}\cup \{0\}, & {\rm if}\ \ m\ge 1,
\end{cases}
\end{equation}
where $\deg(f):=\max\{\deg(X_i)\}_{i=1}^t$, for $f=a_1X_1+\cdots a_tX_t$. Notice that if $R$ has the trivial positive filtration, then $\deg(r)=0$, for all $r\in R$. Thus, $\deg(f) = {\rm tdeg}(f)$ (${\rm tdeg}(f)$ as in Definition \ref{def.total degree}(4) and
$\deg(f)$ as in (\ref{eq1.3.1a})), i.e. the filtration (\ref{eq.filtr A}) coincides with the filtration (\ref{eq1.3.1a}).
More precisely,
\begin{proposition}\label{prop.relat filtr Lezama}
If $A=\sigma(R)\langle x_1,\dots, x_n\rangle$ be a skew PBW extension over $R$ with filtration as in {\rm (\ref{eq1.3.1a})}, then $A$ is a $\sigma$-filtered skew PBW extension if and only if $R$ has the trivial positive filtration.
\end{proposition}
\begin{proof}
Suppose that $A=\sigma(R)\langle x_1,\dots, x_n\rangle$ is $\sigma$-filtered with the filtration given in (\ref{eq1.3.1a}). Then $R$ is positively filtered and $\mathcal{F}'_0(A)=R$. By (\ref{eq.filtr A}), $\mathcal{F}'_0(A)=\{f\in A\mid  {\rm tdeg}(f)\le 0\}\cup \{0\}=R$. Let $r\in R=\mathcal{F}'_0(A)$. Then ${\rm tdeg}(r)=\deg(r)=0= \min\{p\in \mathbb{N}\mid r\in\mathcal{F}_p(R)\}$, i.e. $r\in \mathcal{F}_0(R)$. Therefore $\mathcal{F}_0(R)=R$ and so $\mathcal{F}_p(R)=R$ for $p\geq 0$.

\vspace{0.2cm}

For the converse, assume that $R$ has the trivial positive filtration, i.e. $\mathcal{F}_p(R)=R$ for all $p\geq 0$. Notice that $\sigma_i(\mathcal{F}_p(R))=\sigma_i(R)\subseteq R$ and $\delta_i(\mathcal{F}_p(R))= \delta_i(R)\subseteq R = \mathcal{F}_{p+1}(R)$. Thus $\sigma_i$ and $\delta_i$ are filtered. Now, as $x_jx_i=c_{i,j}x_ix_j+r_{0_{j,i}} + r_{1_{j,i}}x_{1} + \cdots + r_{n_{j,i}}x_{n}$ with
$c_{i,j},r_{0_{j,i}},  r_{1_{j,i}}, \dots, r_{n_{j,i}}\in R=R_0$, then $\deg(c_{i,j})=\deg(r_{0_{j,i}})=\deg(r_{1_{j,i}})= \cdots=
\deg(r_{n_{j,i}})=0$, so, ${\rm tdeg}(c_{i,j}x_ix_j+r_{0_{j,i}} + r_{1_{j,i}}x_{1} + \cdots + r_{n_{j,i}}x_{n})=2$, that is, $A$ preserves ${\rm tdeg}$. Thus $A$ is a $\sigma$-filtered skew PBW extension.
\end{proof}
\begin{remark}\label{rem.general fil Lez}
From Proposition \ref{prop.relat filtr Lezama} we obtain that every skew PBW extension $A$ over a ring $R$ is $\sigma$-filtered with the
filtration given in Lezama and Reyes \cite[Theorem 2.2]{LezamaReyes}. Of course, graded skew PBW extensions  are
trivially $\sigma$-filtered skew PBW extensions. For some examples of skew PBW extensions, filtrations (\ref{eq.filtr A}) and
(\ref{eq1.3.1a}) coincide. More exactly, if we consider a skew PBW extension $A$ over $R$ where $\sigma_i$ is the identity map on $R$ and $\delta_i=0$, for each $1\leq i\leq n$, where $\sigma_i$ and $\delta_i$ are as in Proposition \ref{sigmadefinition}, then the skew PBW extensions are trivially $\sigma$-filtered skew PBW extensions. For instance, if $\mathfrak{g}$ is a finite dimensional Lie algebra over $\mathbb{K}$, then its universal enveloping algebra $U(\mathfrak{g})$ satisfies these conditions. 
\end{remark}
\begin{example}\label{ex.Weyl extend}
\begin{enumerate}
    \item [\rm (1)] The Weyl algebra $\mathbb{K}[t_1,\dots, t_n][x_1,\partial/\partial t_1]\cdots[x_n,\partial/\partial t_n]$, denoted by $A_n(\mathbb{K})$, is a skew PBW extension over $\mathbb{K}[t_1,\dots, t_n]$, i.e. $A_n(\mathbb{K})\cong \sigma(\mathbb{K}[t_1,\dots, t_n])\langle x_1,\dots, x_n\rangle$, where $x_it_j = t_jx_i+ \delta_{ij}$, $x_ix_j-x_jx_i= 0$, and $\delta_{ij}= 0$ for $i\neq j$ and $\delta_{ii} = 1$, $1 \leq i,j \leq n$. The endomorphisms and derivations of Proposition \ref{sigmadefinition} are $\sigma_i$, the identity map of $R$, and $\delta_i=\delta_{ij}$, respectively. If $R$ is endowed with the standard filtration, then $A_n(\mathbb{K})$ is a $\sigma$-filtered skew PBW extension.
\item [\rm (2)] Let $\mathbb{K}$ be a field of characteristic zero. It is well known that $A_n(\mathbb{K})\cong U(\mathfrak{g})/(1 - y) U(\mathfrak{g})$, where $U(\mathfrak{g})$ is the universal enveloping algebra of the $(2n + 1 )$-dimensional Heisenberg Lie algebra with basis given by the set $\{t_1,\dots, t_n,x_1,\dots,x_n, y\}$ over $\mathbb{K}$. Li and Van Oystaeyen \cite[Example (i)]{Li1992} proved that $U(\mathfrak{g})$ is the Rees algebra of $A_n(\mathbb{K})$ with respect to the filtration $\mathcal{F}''$ on $A_n(\mathbb{K})$, where $\mathcal{F}''_p(A_n(\mathbb{K}))=\{\sum_{|\alpha|\leq
p}f_{\alpha}(t_1,\dots,t_n)x^{\alpha}\}$, $f_{\alpha}(t_1,\dots,t_n)\in \mathbb{K}[t_1,\dots,t_n]$ and $x^{\alpha}\in {\rm Mon}(A_n(\mathbb{K}))$. Notice that $\mathcal{F}''$ coincides with the filtration given in (\ref{eq1.3.1a}). Thus, by Proposition \ref{prop.relat filtr Lezama}, $A_n(\mathbb{K})\cong U(\mathfrak{g})/(1-y) U(\mathfrak{g})$ is $\sigma$-filtered with the filtration $\mathcal{F}''$.
\item [\rm (3)] Let $R = \mathbb{K}\langle t_1, t_2\rangle/(t_2t_1-t_1t_2-t_1^2)$ be the Jordan plane and $A=\sigma(R)\langle x_1\rangle$ be the skew PBW extension over $R$, where $x_1t_1=t_1x_1$ and $x_1t_2=t_2x_1+2t_1x_1$. According to Proposition \ref{sigmadefinition} we have $\sigma_1(t_1)=t_1$ and $\sigma_1(t_2) = 2t_1 + t_2$ and
$\delta_1=0$. Notice that $A$ is a $\sigma$-filtered skew PBW extension over the Jordan plane $R$, when $R$ is endowed with the standard filtration.
\item [\rm (4)] Let $R=\mathbb{K}[t]$ be the polynomial ring and $A=\sigma(R)\langle x_1\rangle$ a skew PBW extension over $R$, where $x_1t = c_1tx_1+c_2x_1+c_3t^2+c_4t+c_5$, for $c_1, c_2, c_3, c_4, c_5\in \mathbb{K}$ and $c_1\neq 0$. According to Proposition \ref{sigmadefinition}, $\sigma_1(t)= c_1t+c_2$ and
$\delta_1(t)= c_3t^2+c_4t+c_5$. Therefore $A$ is a $\sigma$-filtered skew PBW extension over $\mathbb{K}[t]$, when $\mathbb{K}[t]$ is endowed with the standard filtration. Notice that since $c_1\neq 0$ then $\sigma_1$ is an automorphism of $\mathbb{K}[t]$.
\end{enumerate}
\end{example}
\begin{remark}\label{rem.no filter}
Let $A$ be the free algebra generated by $x,y$ subject to the relation $yx = xy + x^3$, that is, $\mathbb{K}\langle x, y\rangle/ \langle yx-xy-x^3\rangle$. As one can check (following the ideas presented in \cite{AcostaLezama2015}), $A$ is a skew PBW extension over $\mathbb{K}[x]$. By Proposition \ref{sigmadefinition}, $yx=\sigma(x)y+\delta(x)$, whence $\sigma$ is the identity map of $\mathbb{K}[x]$ and $\delta(x)=x^3$. Notice that the standard filtration of $\mathbb{K}[x]$ is connected, $\sigma$ is filtered but $\delta$ is not filtered. In this way, $A$ is not $\sigma$-filtered.
\end{remark}
Let $R$ be a graded algebra. A graded $R$-module $M$ is \emph{free-graded} on the basis $\{e_j\mid j\in J\}$ if $M$ is free as a left $R$-module on the basis $\{e_j\}$, and also every $e_j$ is homogeneous, say of degree $d(j)$. A filtered module $M=\bigcup_{p\in \mathbb{N}}\mathcal{F}_p(M)$ over a filtered algebra $R=\bigcup_{p\in \mathbb{N}}\mathcal{F}_p(R)$ is \emph{free-filtered} with filtered basis $\{e_j\mid j\in J\}$ if $M$ is a free $R$-module with basis $\{e_j\mid j\in J\}$, and $\mathcal{F}_p(M) = \bigoplus_j\mathcal{F}_{p-p(j)}(R)e_j$, where $p(j)$ is the degree of $e_j$.
\begin{proposition}\label{prop.libre-filt}
If $A=\sigma(R)\langle x_1,\dots,x_n\rangle$ is a $\sigma$-filtered skew PBW extension, then $A$ is free-filtered with filtered basis ${\rm Mon}(A)$. Moreover, $G(A)$ is free-graded over $G(R)$.
\end{proposition}
\begin{proof}
By Definition \ref{def.skewpbwextensions} (2), we have that $A$ is a free $R$-module with basis ${\rm Mon}(A)= \{x^{\alpha}=x_1^{\alpha_1}\cdots x_n^{\alpha_n}\mid \alpha=(\alpha_1,\dots ,\alpha_n)\in \mathbb{N}^n\}$. If the degree of $x^{\alpha}$ is denoted by $\deg(x^{\alpha}):=|\alpha|$, the idea is to show that 
\begin{equation}\label{eq.fil libre fil}
\mathcal{F}_p(A)= \bigoplus_{\alpha\in
\mathbb{N}^n}\mathcal{F}_{p-|\alpha|}(R)x^{\alpha}.
\end{equation}
Let $0\neq f\in \mathcal{F}_p(A)$. By Remark \ref{rem.grade of f} (2), $f$ has a unique representation given by $f=c_1X_1+\cdots+c_dX_d$, with $c_i\in R\ \backslash\ \{0\}$ and $X_i:=x^{\alpha^{i}}= x_1^{\alpha^i_{1}}\cdots x_n^{\alpha^i_{n}}\in {\rm Mon}(A)$,  where ${\rm tdeg}(c_iX_i) = {\rm tdeg}(c_ix^{\alpha^{i}})=\deg(c_i)+\deg(X_i)=\deg(c_i)+|\alpha^{i}|\leq p$, for $1\leq i\leq d$. Hence, $\deg(c_i)\leq p-|\alpha^{i}|$, i.e. $c_i\in \mathcal{F}_{ p-|\alpha^{i}|}(R)$, for $1\leq i\leq d$, and so $f\in \mathcal{F}_{p-|\alpha^{1}|}(R)x^{\alpha^1}+\cdots+\mathcal{F}_{p-|\alpha^{d}|}(R)x^{\alpha^d}$. From the uniqueness of the representation of $f$, it follows that $f\in \bigoplus_{\alpha\in \mathbb{N}^n}\mathcal{F}_{p-|\alpha|}(R)x^{\alpha}$. 

\vspace{0.2cm}

For the other inclusion, let $f\in \bigoplus_{\alpha\in \mathbb{N}^n}\mathcal{F}_{p-|\alpha|}(R)x^{\alpha}$. Then $f$ has a unique representation as $f= f_{\alpha^1}+\cdots + f_{\alpha^t}$, where $f_{\alpha^j}\in \mathcal{F}_{p-|\alpha^j|}(R)x^{\alpha^j}$, for $1\leq j\leq t$. Thus, $f_{\alpha^j}= r_{j}x^{\alpha^j}$, with
$r_j\in \mathcal{F}_{p-|\alpha^j|}$ and $x^{\alpha^j}\in {\rm Mon}(A)$. In this way, $\deg(r_j)\leq p-|\alpha^j|$, and therefore ${\rm tdeg}(f_{\alpha^j})=\deg(r_j)+|\alpha^j|\leq p$, whence ${\rm tdeg}(f)\leq p$, that is, $f\in \mathcal{F}_p(A)$. This means that $A$ is free-filtered with filtered basis ${\rm Mon}(A)$. Finally, since $A$ free-filtered, by \cite[Proposition 7.6.15]{McConnell}, we obtain that $G(A)$ is free-graded over $G(R)$ on the graded basis $\overline{{\rm Mon}(A)}$.
\end{proof}
\section{The homogenization of a \texorpdfstring{$\sigma -$}\ filtered skew PBW extension}\label{sect.homogeniza}
\begin{proposition}\label{prop.R f. pres impl A f. pres}
Let $A=\sigma(R)\langle x_1,\dots,x_n\rangle$ be a skew PBW extension over an algebra $R$.
\begin{enumerate}
\item[\rm (1)] If $R$ is finitely generated as algebra, then $A$ is finitely generated as algebra.
\item[\rm (2)] If $R$ is finitely presented, then $A$ is finitely
presented.
\end{enumerate}
\end{proposition}
\begin{proof} 
\begin{enumerate}
\item[\rm (1)] If $R$ is finitely generated as algebra, then there exists a finite set of elements $t_1,\dots, t_s\in R$ such that the set $\{t_{i_1}t_{i_2}\cdots t_{i_m}\mid 1\leq i_j\leq s, m\geq 1\} \cup \{1\}$ spans $R$ as a $\mathbb{K}$-space. By Definition \ref{def.skewpbwextensions} (2), ${\rm Mon}(A)=
\{x^{\alpha}=x_1^{\alpha_1}\cdots x_n^{\alpha_n}\mid \alpha=(\alpha_1,\dots ,\alpha_n)\in \mathbb{N}^n\}$ is an $R$-basis for $A$. There exists a finite set of elements $t_1,\dots, t_s, x_1, \dots, x_n\in A$ such that the set $\{t_{i_1}t_{i_2}\cdots t_{i_m}x_1^{\alpha_1}\cdots x_n^{\alpha_n}\mid 1\leq i_j\leq s, m\geq 1,\alpha_1,\dots, \alpha_n\in \mathbb{N}\}$ spans $A$ as a $\mathbb{K}$-space.

\item[\rm (2)] If $R$ is finitely presented, then $R= \mathbb{K}\langle
t_1,\dots, t_m\rangle/I$, where
    \begin{equation}\label{rel1.R}
I=\langle r_1,\dots, r_s\rangle
\end{equation}
is a two-sided ideal of $\mathbb{K}\langle t_1, \dots, t_m\rangle$
generated by a finite set $r_1,\dots, r_s$ of polynomials in
$\mathbb{K}\langle t_1, \dots, t_m\rangle$. In this way,
\begin{align}\label{eq1.repr of A}
\begin{split}
A= &\ \mathbb{K}\langle t_1, \dots, t_m, x_{1}, \dots,
x_{n}\rangle/J,
\text{ where}\\
J= &\ \langle r_1,\dots, r_s, \ f_{ik}, \ g_{ji} \mid 1\leq i,j\leq
n,\  1\leq k\leq m\rangle
\end{split}
\end{align}
  is the two-sided ideal of $\mathbb{K}\langle t_1, \dots, t_m, x_{1}, \dots, x_{n}\rangle$ generated by a finite set of polynomials  $r_1,\dots, r_s$,
$f_{ik}$, $g_{ji}$ with $r_1,\dots, r_s$ as in (\ref{rel1.R}), that is, 
 \begin{equation}\label{eq1.rel xt}
 f_{ik}:= x_{i}t_k-\sigma_{i}(t_k)x_{i}-\delta_{i}(t_k)
 \end{equation}
 where $\sigma_{i}$ and $\delta_{i}$ are as in Proposition \ref{sigmadefinition}, i.e.
  \begin{equation}\label{eq1.rel xx}
 g_{ji}:= x_{j}x_{i}-c_{i,j}x_{i}x_{j} - (r_{0_{j,i}} + r_{1_{j,i}}x_{1} + \cdots + r_{n_{j,i}}x_{n})
 \end{equation}
 as in  (\ref{eq.rep variables x}). Therefore, $A$ is finitely
 presented.
 \end{enumerate}
\end{proof}
\begin{proposition}\label{prop.filtracion fin pres}
If $A$ is a skew PBW extension over a finitely presented algebra $R$ such that $\sigma_i$, $\delta_i$ are filtered, and $A$ preserves ${\rm tdeg}$, then $A$ is a connected $\sigma$-filtered algebra, and the filtration of $A$ is finite.
\end{proposition}
\begin{proof}
Let $A=\sigma(R)\langle x_1,\dots, x_n\rangle$ be a skew PBW extension over a finitely presented algebra $R= \mathbb{K}\langle t_1,\dots, t_m\rangle/\langle r_1,\dots, r_s\rangle$. As
$L=\mathbb{K}\langle t_1,\dots, t_m\rangle$ is connected filtered, $R$ inherits a connected filtration $\{\mathcal{F}_q(R)\}_{q\in \mathbb{N}}$, from the standard filtration on the free algebra $\mathbb{K}\langle t_1,\dots, t_m\rangle$. Since $\sigma_i$ and $\delta_i$ are filtered and $A$ preserves ${\rm tdeg}$, then $A$ is $\sigma$-filtered. Now, as $R$ is connected filtered, Proposition \ref{prop.connected filt} implies that $A$ is connected filtered. Notice that
$L_p$ is a finite dimensional subspace of $L$, for all $p\in
\mathbb{N}$. In this way, $\mathcal{F}_q(L)$ is also a finite dimensional
subspace, and so $\mathcal{F}_q(R)$ is a finite dimensional subspace
of $R$, for all $q\in \mathbb{N}$, i.e. the filtration of $R$ is finite. Proposition \ref{prop.filt finita} (2) guarantees that the filtration of $A$ is finite.
\end{proof}
\begin{theorem}\label{teo. homogenization} If $A=\sigma(R)\langle x_1,\dots, x_n\rangle$ is a
bijective skew PBW extension over a finitely presented algebra $R$ such that $\sigma_i$, $\delta_i$ are filtered and $A$ preserves ${\rm tdeg}$, then $H(A)$ is a graded skew PBW extension over $H(R)$.
\end{theorem}
\begin{proof}
Consider the ring $R= \mathbb{K}\langle t_1,\dots,
t_m\rangle/\langle r_1,\dots, r_s\rangle= L/\langle r_1,\dots,
r_s\rangle$ and let $H(R) = L[z]/\langle\widehat{r}_1, \dots,
\widehat{r}_s\rangle=\mathbb{K}\langle t_1,\dots,
t_m,z\rangle/\langle \widehat{r}_1, \dots, \widehat{r}_s,\
t_kz-zt_k\mid 1\leq k\leq m\rangle$ be the homogenization of $R$. By Proposition \ref{prop.R f. pres impl A f. pres} (2) and its proof, $A$ is finitely presented with presentation
\begin{align}\label{eq. pres completa A}
A = \frac{\mathbb{K}\langle t_1, \dots, t_m, x_{1}, \dots,
x_{n}\rangle}{\langle r_1,\dots, r_s,\ \ f_{ik}, \ g_{ji}\mid 1\leq
i,j\leq n,\  1\leq k\leq m\rangle},
\end{align}
where $\ f_{ik}= x_{i}t_k-\sigma_{i}(t_k)x_{i}-\delta_{i}(t_k)$,
$g_{ji}= x_{j}x_{i}-c_{i,j}x_{i}x_{j} - (r_{0_{j,i}} +
r_{1_{j,i}}x_{1} + \cdots + r_{n_{j,i}}x_{n})$, with
$c_{i,j},r_{0_{j,i}}, r_{1_{j,i}}, \dots, r_{n_{j,i}}\in R$. Let
$L_{tx}:=\mathbb{K}\langle t_1, \dots, t_m, x_{1}, \dots,
x_{n}\rangle$. Then
\begin{align}\label{eq.rep H(A)}
H(A) = &\ L_{tx}[z]/\langle\widehat{r}_1, \dots, \widehat{r}_s,\
\widehat{f_{ik}}, \ \widehat{g_{ji}}\mid 1\leq i,j\leq n,\  1\leq
k\leq m\rangle \\ = &\ \frac{\mathbb{K}\langle z, t_1, \dots, t_m,
x_{1}, \dots, x_{n}, \rangle}{\langle\widehat{r}_1, \dots,
\widehat{r}_s,\ \widehat{f_{ik}}, \ \widehat{g_{ji}}, t_kz-zt_k,
x_iz-zx_i\mid 1\leq i,j\leq n,\  1\leq k\leq m\rangle}.
\end{align}
By Proposition \ref{prop.filtracion fin pres}, $A$ is a connected $\sigma$-filtered algebra, whence $\sigma_{i}$ and $\delta_{i}$ are
filtered and $A$ preserves ${\rm tdeg}$. Hence,
\begin{align*}
{\rm tdeg}(f_{ik})&=
{\rm tdeg}(x_{i}t_k-\sigma_{i}(t_k)x_{i}-\delta_{i}(t_k))=2,\\
{\rm tdeg}(g_{ji})&= {\rm tdeg}(x_{j}x_{i}-c_{i,j}x_{i}x_{j} -
(r_{0_{j,i}} + r_{1_{j,i}}x_{1} + \cdots + r_{n_{j,i}}x_{n})) = 2.
\end{align*}

\vspace{0.2cm}

Notice that 
\begin{align}\label{eq.rep f gorro}
\widehat{f_{ik}}&=x_{i}t_k-(\sigma_{i}(t_k)z^{1-\deg(\sigma_{i}(t_k))})x_{i}-\delta_{i}(t_k)z^{2-\deg(\delta_{i}(t_k))},\\
\widehat{g_{ji}}&=x_{j}x_{i}-c_{i,j}x_{i}x_{j} -
r_{0_{j,i}}z^{2-\deg(r_{0_{j,i}})} -
(r_{1_{j,i}}z^{1-\deg(r_{1_{j,i}})})x_{1} - \cdots -
(r_{n_{j,i}}z^{1-\deg(r_{n_{j,i}})})x_{n},
\end{align}
where $c_{i,j}\in \mathbb{K}\setminus\{0\}$,
$r_{0_{j,i}}z^{2-\deg(r_{0_{j,i}})},\
r_{1_{j,i}}z^{1-\deg(r_{1_{j,i}})}, \dots,
r_{n_{j,i}}z^{1-\deg(r_{n_{j,i}})}\in H(R)$, for $1\leq i,j\leq
n$.

\vspace{0.2cm}

From (\ref{eq.rep f gorro}), in $H(A)$ we have the relations 
\begin{align}\label{eq.relac iii en H(A)}
x_{i}t_k=(\sigma_{i}(t_k)z^{1-\deg(\sigma_{i}(t_k))})x_{i}+\delta_{i}(t_k)z^{2-\deg(\delta_{i}(t_k))}
\end{align}
and
\begin{align}\label{eq.relac iv en H(A)}
 x_{j}x_{i}-c_{i,j}x_{i}x_{j} =
r_{0_{j,i}}z^{2-\deg(r_{0_{j,i}})} +
(r_{1_{j,i}}z^{1-\deg(r_{1_{j,i}})})x_{1} + \cdots +
(r_{n_{j,i}}z^{1-\deg(r_{n_{j,i}})})x_{n},
\end{align}
where $c_{i,j}\in \mathbb{K}\ \backslash\ \{0\}$,
$r_{0_{j,i}}z^{2-\deg(r_{0_{j,i}})},\
r_{1_{j,i}}z^{1-\deg(r_{1_{j,i}})}, \dots,
r_{n_{j,i}}z^{1-\deg(r_{n_{j,i}})}\in H(R)$, for $1\leq i,j\leq
n$.

\vspace{0.2cm}

Relations given in (\ref{eq.relac iii en H(A)}) and (\ref{eq.relac iv en H(A)}) correspond to Definition \ref{def.skewpbwextensions} (3) and (4), respectively, applied to $H(R)$ and $H(A)$. Notice that $H(R)\subseteq H(A)$ and $H(A)$ is an $H(R)$-free module with basis
\begin{center}
${\rm Mon}(A):= {\rm
Mon}\{x_1,\dots,x_n\}:=\{x^{\alpha}=x_1^{\alpha_1}\cdots
x_n^{\alpha_n}\mid\alpha=(\alpha_1,\dots ,\alpha_n)\in
\mathbb{N}^n\}$,
\end{center}
whence $H(A)$ is a skew PBW extension over $H(R)$ in the variables $x_1,\dots, x_n$. Thus $H(A)\cong\sigma(H(R))\langle x_1,\dots,
x_n\rangle$. Notice that $H(R)$ is a connected graded algebra, i.e.
$R(H)=\bigoplus_{p\geq 0}H(R)_p$. From (\ref{eq.relac iii en H(A)})
and applying Proposition \ref{sigmadefinition} to $H(A)$, it follows
that $\widehat{\sigma_i}(t_k)=
(\sigma_{i}(t_k)z^{1-\deg(\sigma_{i}(t_k))})$;
$\widehat{\sigma_i}(z)=z$; $\widehat{\delta_i}(t_k)=
\delta_{i}(t_k)z^{2-\deg(\delta_{i}(t_k))}$ and
$\widehat{\delta_i}(z)=0$, for $1\leq i\leq n$. Thus,
$\widehat{\sigma_i}: H(R)\to H(R)$ is a graded ring homomorphism and
$\widehat{\delta_i}: H(R)(-1)\to H(R)$ is a graded
$\widehat{\sigma_i}$-derivation for all $1\leq i \leq n$. Now, from
(\ref{eq.relac iv en H(A)}) we have that $c_{i,j}\in
\mathbb{K}\setminus\{0\}$, $r_{0_{j,i}}z^{2-\deg(r_{0_{j,i}})}\in
H(R)_2$,\ $r_{1_{j,i}}z^{1-\deg(r_{1_{j,i}})}, \dots,
r_{n_{j,i}}z^{1-\deg(r_{n_{j,i}})}\in H(R)_1$, for $1\leq i,j\leq
n$. Thus, $x_jx_i-c_{i,j}x_ix_j\in H(R)_2+H(R)_1x_1 +\cdots +
H(R)_1x_n$, and $c_{i,j}\in H(R)_0=\mathbb{K}$. As $\sigma_i$ is
bijective then $\widehat{\sigma_i}$ is bijective and as $c_{i,j}$ is
invertible, then $H(A)$ is bijective. Therefore, $H(A)=\sigma(H(R))\langle x_1,\dots, x_n\rangle$ is a bijective skew PBW extension over the $\mathbb{N}$-graded algebra $R(H)$ that satisfies both conditions formulated in Proposition \ref{prop.grad A}, and so $H(A)$ is a graded skew PBW extension over $H(R)$.
\end{proof}
\begin{remark}\label{rem.pres f in H(A)} Let $A$ as in Theorem \ref{teo. homogenization}. By (\ref{eq.rep
H(A)}),
\[
H(A) = \frac{\mathbb{K}\langle z, t_1, \dots, t_m, x_{1}, \dots,
x_{n}, \rangle}{\langle\widehat{r}_1, \dots, \widehat{r}_s,\
\widehat{f_{ik}}, \ \widehat{g_{ji}}, t_kz-zt_k, x_iz-zx_i\mid 1\leq
i,j\leq n,\  1\leq k\leq m\rangle}.
\]
Let $f\in H(A)$. By Theorem \ref{teo. homogenization}, $H(A)=\sigma(H(R))\langle x_1,\dots, x_n\rangle$ is a graded skew PBW extension over $H(R)$. Remark \ref{rem.grade of f} (2) shows that $f$ has a unique representation as $f=c_1X_1+\cdots+c_dX_d$,\  with
\ $c_i\in H(R)\ \backslash\ \{0\}$ and  $X_i\in {\rm Mon}(A)$. As
$H(A)$ is graded, $c_iX_i$ is homogeneous in $H(A)$, let us say of
degree $d(i)$, and $c_i$ is homogeneous in $H(R)$. Then $c_iX_i=k_iz^{\beta_i}\overline{t_{i_1}t_{i_2}\dots\ t_{i_p}}X_i$, $0\leq i_j\leq m$, where $\overline{t_{i_1}t_{i_2}\dots\ t_{i_p}}\in R = \mathbb{K}\langle t_1, \dots, t_m\rangle/\langle r_1,\dots, r_s\rangle$, $\overline{t_0}=x_0=z^0=1$, $k_i\in \mathbb{K}$ and $\beta_i+p+\deg(X_i)=d(i)$. Therefore $f$ has a unique representation as $f=k_1z^{\beta_1}\overline{t_{1_1}t_{1_2}\dots\ t_{1_p}}X_1+\cdots+ k_dz^{\beta_d}\overline{t_{d_1}t_{d_2}\dots\ t_{d_p}}X_d$.
\end{remark}
Let $A$ be a bijective $\sigma$-filtered skew PBW extension over a finitely presented algebra $R$. Let us fix the notation: $G(A)$ and ${\rm Rees}(A)$ is respectively the associated graded algebra  and the associated Rees algebra of $A$ regarding the filtration given in Theorem \ref{teo.filtracion gen}. One can always recover $A$ and $G(A)$ from $H(A)$, via $A\cong H(A)/\langle z-1\rangle$ and $G(A) \cong H(A)/\langle z\rangle$, respectively.
\begin{proposition}\label{prop.rel homogeniz Grad}
If $A=\sigma(R)\langle x_1,\dots, x_n\rangle$ is a bijective $\sigma$-filtered skew PBW extension over a finitely presented algebra $R$, then the following assertions hold:
\begin{enumerate}
\item[\rm (1)] $A\cong H(A)/\langle z-1\rangle$.
\item[\rm (2)] $H(A)/\langle z\rangle\cong G(A)$.
\end{enumerate}
\end{proposition}
\begin{proof}
(1) It is clear.

(2) From (\ref{eq1.repr of A}) and (\ref{eq. pres completa A}),
\begin{align*}
A = &\ \mathbb{K}\langle t_1, \dots, t_m, x_{1}, \dots,
x_{n}\rangle/J \\ 
= &\ \mathbb{K}\langle t_1, \dots, t_m, x_{1}, \dots,
x_{n}\rangle/\langle r_1,\dots, r_s,\ \ f_{ik}, \ g_{ji}\mid 1\leq i,j\leq n,\  1\leq k\leq m\rangle, \end{align*} 
i.e. $A$ is an algebra defined by generators and relations. Then there is a
standard connected filtration $\{\mathcal{F}^*_p(A)\}_{p\in
\mathbb{N}}$ on $A$ wherein
$\mathcal{F}^*_p(A)=(\mathcal{F}_p(L_{tx})+J)/J$, i.e.
$\mathcal{F}^*_p(A)$ is the span of all words in the variables $t_1,
\dots, t_m, x_{1}, \dots, x_{n}$ of degree at most $p$. Notice that
for $\sigma$-filtered skew PBW extensions of a finitely presented
algebra $R$, $\mathcal{F}^*_p(A)$ coincides with $\mathcal{F}_p(A)$
as in (\ref{eq.filtr A}), for all $p\geq 0$. Therefore,
$H(A)/\langle z\rangle\cong G(A)$.
\end{proof}
\begin{example}\label{ex-homogWeyl} The Weyl algebra $A_n(\mathbb{K})$ in Example \ref{ex.Weyl
extend} is  the free associative algebra with generators $t_1,\dots,
t_n$, $x_1,\dots, x_n$ modulo the relations $t_jt_i=t_it_j$,\quad
$x_jx_i = x_ix_j$,\quad $x_it_j =t_jx_i+\delta_{ij}$, where
$\delta_{ij}$ is the Kronecker delta, $1\leq i, j \leq n$. Adding
another generator $z$ that commutes with $t_1,\dots, t_n$,
$x_1,\dots, x_n$ and replacing $\delta_{ij}$ with $\delta_{ij}z^2$
in the above relations yields the homogenized Weyl algebra
$H(A_n(\mathbb{K}))$. Since $A_n(\mathbb{K})$ is a bijective skew PBW extension over a finitely presented algebra $R = \mathbb{K}[t_1,\dots, t_n]$, such that $\sigma_i$, $\delta_i$ are
filtered and $A_n(\mathbb{K})$ preserves {\rm tdeg}, by Theorem
\ref{teo. homogenization} we have that
$H(A_n(\mathbb{K}))$ is a graded skew PBW extension over $H(R)=\mathbb{K}[t_1,\dots, t_n, z]$. By Proposition \ref{prop.rel
homogeniz Grad}, $A_n(\mathbb{K})\cong H(A_n(\mathbb{K}))/\langle
z-1\rangle$ and  $H(A_n(\mathbb{K}))/\langle z\rangle\cong
G(A_n(\mathbb{K}))$, which is just a commutative polynomial ring in
2n variables, where $G(A_n(\mathbb{K}))$ is the associated graded
algebra of $A_n(\mathbb{K})$ with respect to the filtration
$\mathcal{F}$ as in (\ref{eq.filtr A}).
\end{example}
\begin{example}\label{ex-tresdimensional}
Following \cite{BellSmith1990} or \cite[Definition C4.3]{Rosenberg}, a $3$-{\em dimensional algebra} $\mathcal{A}$ is a $\mathbb{K}$-algebra generated by the indeterminates $x, y, z$ subject to the relations $yz-\alpha zy=\lambda$, $zx-\beta xz=\mu$, and $xy-\gamma yx=\nu$, where $\lambda,\mu,\nu \in \mathbb{K}x+\mathbb{K}y+\mathbb{K}z+\mathbb{K}$, and $\alpha, \beta, \gamma \in \mathbb{K}^{*}$. $\mathcal{A}$ is called a \textit{3-dimensional skew polynomial $\mathbb{K}$-algebra} if the set $\{x^iy^jz^k\mid i,j,k\geq 0\}$ forms a $\mathbb{K}$-basis of the algebra. As we can see in \cite[Theorem C.4.3.1]{Rosenberg}, there are exactly fifteen non-isomorphic 3-dimensional skew polynomial $\mathbb{K}$-algebras, and one of these is the enveloping algebra of $\mathfrak{sl}(2,\mathbb{K})$. From the definition, it is clear that these algebras are skew PBW extensions over $\mathbb{K}$, and hence, $\sigma$-filtered skew PBW extensions. 

\vspace{0.2cm}

Redman \cite{Redman1999} studied the geometry of the homogenizations of two classes of 3-dimensional skew polynomial algebras. Following the terminology used in \cite{BellSmith1990}, the algebras Type I and Type II (these objects are called like this because these are two classes of three dimensional skew polynomial rings that have finite dimensional simple modules of arbitrarily large dimensions) are defined as 
\[
{\rm Type\ I}\ \begin{cases}
    g_1 = yz - \alpha zy,\\
    g_2 = zx - \beta xz - ay - b,\\
    g_3 = xy - \alpha yx
\end{cases}\ \ {\rm and}\ \ \ \ {\rm Type\ II}\ \begin{cases}
    g_1 = yz - \alpha zy - x - b_1,\\
    g_2 = zx - \alpha xz - y - b_2,\\
    g_3 = xy - \alpha yx - z - b_3
\end{cases}
 \]
where $a, b, b_1, b_2, b_3, \alpha, \beta \in \mathbb{C}$ with such that $\alpha\beta \neq 0$.

\vspace{0.2cm}

The homogenization
$H(\mathcal{A})$ of both types of algebras with respect to a central variable $t$ is given by $\mathbb{C}\langle
x,y,z,t\rangle$ with defining relations
 \[
 \begin{cases}
 yz - \alpha zy = 0,\\
    zx - \beta xz = ayt + bt^{2},\\
    xy - \alpha yx = 0,
 \end{cases}\ \ {\rm or}\ \ \ \ \begin{cases}
 yz - \alpha zy = xt + b_1t^{2},\\
    zx - \alpha xz = yt + b_2t^{2},\\
    xy - \alpha yx = zt + b_3t^2,
 \end{cases}
 \]
 and 
 $xt - tx = yt - ty = zt - tz = 0$. From \cite[Proposition 2.1.1]{BellSmith1990}, the standard monomials $\{x^{i}y^{j}z^{k}t^{l}\mid i, j, k, l\ge 0\}$ form a $\mathbb{C}$-basis for the algebra $D$ with the degree and dictionary ordering being $x > y > z > t$, whence $t$ is a non-zero
divisor.

\vspace{0.2cm}

By Proposition \ref{prop.rel homogeniz Grad}, $\mathcal{A}\cong
H(\mathcal{A})/\langle t-1\rangle$ and $H(\mathcal{A})/\langle t\rangle\cong
G(\mathcal{A})$, where $G(\mathcal{A})$ is the associated graded algebra of $\mathcal{A}$
with respect to the filtration $\mathcal{F}$ as in (\ref{eq.filtr
A}). Thus, $H(\mathcal{A})$ is a central extension of the algebra $H(\mathcal{A})/\langle t\rangle$, and therefore $H(\mathcal{A})$ is a central extension of $G(\mathcal{A})$. These facts were used in \cite{Redman1999} to prove that the quotient algebra $H(\mathcal{A})/\langle t\rangle$ is a 3-dimensional Artin-Schelter regular algebra (\cite[Lemma 1.1]{Redman1999}), $H(\mathcal{A})$ is 4-dimensional Artin-Schelter regular, graded Noetherian domain, and Cohen Macaulay with Hilbert series $(1-t)^{-4}$ (\cite[Proposition 1.2]{Redman1999}). He also described the noncommutative projective geometry of these objects, and compute the finite dimensional simple modules for the homogenization of Type I algebras in the case that $\alpha$ is not a primitive root of unity. In this case, all finite dimensional simple modules are quotients of line modules that are homogenizations of Verma modules. From Theorem \ref{teo. homogenization}, we know that $H(\mathcal{A})$ is a graded skew PBW extension over $\mathbb{C}[t]$. In Theorem \ref{teo.propiedades H(A)} below, we generalize some of these properties for $\sigma$-filtered skew PBW extensions over a ring $R$ such that $H(R)$ is Auslander-regular.
\end{example}
\begin{example}\label{ex-Liealgebratresdimensional}
Following Le Bruyn and Smith \cite{LeBruynSmith1993}, we write
$\mathfrak{g}$ = $\mathbb{C}e\oplus \mathbb{C}f\oplus\mathbb{C}h$ and
define a vector space isomorphism  $\mathfrak{sl}(2, \mathbb{C})\to
\mathfrak{g}$ by
\[
\begin{pmatrix}
0 & 1 \\
0 & 0
\end{pmatrix} \to e,\quad \begin{pmatrix}
0 & 0 \\
1 & 0 
\end{pmatrix} \to f,\quad \begin{pmatrix}
1 & 0 \\
0 & -1 
\end{pmatrix}\to h,
\]
and we transfer the Lie bracket on $\mathfrak{sl}(2, \mathbb{C})$ to $\mathfrak{g}$ giving $[e, f]=h$, $[h,e]=2e$, $[h, f]=-2f$. The homogenization $H(U(\mathfrak{g}))$ of the universal enveloping algebra of $\mathfrak{g}$ with respect to a central variable $t$ is $\mathbb{C}\langle e, f, h, t\rangle$ with defining equations 
\[
ef - fe = ht,\quad he - eh =
2et,\quad hf - fh = -2ft\quad et - te = ft - tf = ht - th = 0.
\]
By Proposition \ref{prop.rel homogeniz Grad}, $U(\mathfrak{g})\cong H(U(\mathfrak{g}))/\langle t-1\rangle$ and $H(U(\mathfrak{g}))/\langle t\rangle\cong G(U(\mathfrak{g}))$, where $G(U(\mathfrak{g}))$ is the associated graded algebra of $U(\mathfrak{g})$ with respect to the filtration $\mathcal{F}$ as in (\ref{eq.filtr A}). Thus, $H(U(\mathfrak{g}))$ is a central extension of $H(U(\mathfrak{g}))/\langle t\rangle\cong G(U(\mathfrak{g}))\cong \mathbb{C}[e, f, h] = S(\mathfrak{g})$, the symmetric algebra on $\mathfrak{g}$. These facts were used in \cite{LeBruynSmith1993} to deduce that $H(U(\mathfrak{g}))$ has Hilbert series $(1-t)^{-4}$, it is a positively graded Noetherian domain, Auslander-regular of dimension 4, satisfies the Cohen-Macaulay property, and its center is $\mathbb{C}[\Omega, t]$, where $\Omega = h^2 + 2ef + 2fe$ is the Casimir element. Recall that $U(\mathfrak{g})$ is a $\sigma$-filtered skew PBW extension (Example \ref{ex.Weyl extend}), and by Theorem \ref{teo. homogenization}, $H(U(\mathfrak{g}))$ is a graded skew PBW extension over $\mathbb{C}[t]$. Below, using Theorem \ref{teo.propiedades H(A)}, we obtain some of the above properties for $\sigma$-filtered skew PBW extensions over $R$ such that $H(R)$ is Auslander-regular. 
\end{example}
\begin{remark}
A graded algebra $R$ is said to be \emph{generated
in degree one if $R_1$ generates $R$ as an algebra.} Let $A$ be a $\sigma$-filtered skew PBW extension over a commutative
polynomial ring $R=\mathbb{K}[t_1,\dots, t_m]$. Notice that if we use
the filtration given in (\ref{eq1.3.1a}), then $H(A)$ and $G(A)$ are
not generated in degree one, while if we use the standard filtration
for $R$, then $H(A)$ and $G(A)$ are generated in degree one. For the
study of certain properties in an algebra such as Artin-Schelter regular,
strongly Noetherian, Auslander regular, Cohen-Macaulay, Koszul and
the Jacobson radical, some authors impose the condition that the
algebra be generated in degree one (e.g. \cite[Theorem 3.3]{Greenfeld2019} or \cite[Theorem 0.1]{Zhang2}). Now, for example presented by Greenfeld et al. \cite[Section 6]{Greenfeld2019}, it was considered that in Ore extensions of endomorphism type $R[x,\sigma]$, $\deg (r) = 0$, for all non-zero $r\in R$ and $\deg (rx^j) = j$, for all natural number $j > 0$. This fact was used to study the properties graded
nilpotent (a graded algebra is \emph{graded nilpotent} if the
algebra generated by any set of homogeneous elements of the same
degree is nilpotent), graded locally nilpotent (a graded algebra
is \emph{graded locally nilpotent} if the algebra generated by any
finite set of homogeneous elements of the same degree is nilpotent),
the Jacobson radical, and to ask some related questions in the Ore
extension $R[x; \sigma]$.
\end{remark}
\section{Other homological properties}\label{sect.essen Cala}
It is known that a graded algebra $R$ is right (left) Noetherian if and only if it is graded right (left) Noetherian, which means that every graded right (left) ideal is finitely generated (\cite[Proposition 1.4]{Levasseur}). Let $M$ be an $R$-module. The \emph{grade number} of $M$ is $j_R(M) := \min\{p\mid {\rm Ext}_R^p(M, R)\neq 0\}$ or $\infty$ if no such $p$ exists. Notice that $j_R(0) = \infty$. When $R$ is Noetherian, $j_R(M)\leq {\rm pd}_R(M)$ (where ${\rm pd}_R(M)$ denotes the projective dimension of $M$), and if furthermore injdim$(R)=q< \infty$, we have $j_R(M)\leq q$, for all non-zero finitely generated $R$-module $M$ (see \cite{Levasseur} for further details).
\begin{definition}(\cite[Definition 2.1]{Levasseur}).\label{def. Ausl Gor Reg}
Let $R$ be a Noetherian ring.
\begin{enumerate}
\item[\rm (1)] An $R$-module $M$ satisfies the \emph{Auslander-condition} if for all $p \geq 0$, $j_R(N)\geq p$, for every $R$-submodule $N$ of ${\rm Ext}_R^p(M, R)$.
\item[\rm (2)] $R$ is called \emph{Auslander-Gorenstein} of dimension $q$ if injdim$(R) = q <\infty$, and every left or right finitely generated $R$-module satisfies the Auslander-condition.
\item[\rm (3)] $R$ is said to be \emph{Auslander-regular} of dimension $q$ if gld$(R) = q <\infty$, and every  left or right  finitely generated $B$-module satisfies the Auslander-condition.
\end{enumerate}
\end{definition}
\begin{theorem}\label{teo.propiedades H(A)}
If $A$ is a $\sigma$-filtered skew PBW extension over a ring $R$ such that
$H(R)$ is Auslander-regular, then the following assertions hold:
\begin{enumerate}
\item[\rm (1)] $H(A)$ is graded Noetherian.
\item[\rm (2)] $H(A)$ is a domain.
\item[\rm (3)] $A$ is Noetherian.
\item[\rm (4)] $A$ is a PBW deformation of $G(A)$.
\item[\rm (5)] $G(A)$ is Noetherian.
\item[\rm (6)] ${\rm Rees}(A)\cong H(A)$.
\item[\rm (7)] $A$ is Zariski.
\item[\rm (8)] $H(A)$ is  Artin-Schelter regular.
\item[\rm (9)] $G(A)$ is Artin-Schelter regular.
\end{enumerate}
\end{theorem}
\begin{proof}
From Theorem \ref{teo. homogenization}, we know that $H(A)$ is a graded bijective skew PBW extension over $H(R)$.
\begin{enumerate}
\item [\rm (1)] Since $H(R)$ is Noetherian and graded, by \cite[Proposition 1.4]{Levasseur} $H(R)$ is graded Noetherian. From \cite[Proposition 2.7]{SuarezLezamaReyes2107-1}, we obtain that $H(A)$ is graded Noetherian.
\item [\rm (2)] By \cite[Theorem 2.9]{LezamaVanegas} we know that $H(A)$ is Auslander-regular. Since $H(R)$ is connected graded, \cite[Remark 2.10]{Suarez} implies that $H(A)$ is connected graded, and thus the assertion follows from \cite[Theorem 4.8]{Levasseur}.
\item [\rm (3)]  Part (1) above shows that $H(R)$ is Noetherian, whence $A\cong H(A)/\langle z-1\rangle$ is Noetherian.
\item [\rm (4)] Notice that $A$ is a deformation of $G(A)$. By (2), $z$ is regular in $H(A)$, and since $G(A)\cong H(A)/\langle z\rangle$, then $H(A)$ is a central regular extension of $G(A)$. By \cite[Theorem 1.3]{Cassidy2007}, $A$ is a PBW deformation of $G(A)$.
\item [\rm (5)] Since $H(A)$ is connected graded and $z\in H(A)_1$ is central
regular, then Proposition \cite[Proposition 3.5]{Levasseur} implies that $H(A)$ is Noetherian if and only if $H(A)\langle z\rangle\cong G(A)$ is Noetherian.
\item [\rm (6)] It follows from (4) and \cite[Proposition 2.6]{Wu2013}.
\item [\rm (7)] From (1) and (6), ${\rm Rees}(A)$ is Noetherian. As $A$ is connected filtered, then $A$ is Zariski.
\item [\rm (8)] Since $H(R)$ is finitely presented connected Auslander-regular and $H(A)$ is a graded skew PBW extension over $H(R)$, then by \cite[Proposition 3.5 (iii)]{SuarezLezamaReyes2107-1} we have that $H(A)$ is Artin-Schelter regular.
\item [\rm (9)] As $H(A)$ is a connected graded, by (2) $H(A)$ is a domain, and by (8)  $H(A)$ is Artin-Schelter regular. From \cite[Corollary 1.2]{Rogalski2012}, $G(A)\cong H(A)/\langle z\rangle$ is Artin-Schelter regular.
\end{enumerate}
\end{proof}
Example \ref{ex.Weyl extend} showed that the Weyl algebra $A_n(\mathbb{K})$ is a $\sigma$-filtered skew PBW extension over $R=\mathbb{K}[t_1,\dots,t_n]$, and by Example \ref{ex-homogWeyl}, $H(R)=\mathbb{K}[t_1,\dots,t_n, z]$, which is Auslander-regular. Hence, Theorem \ref{teo.propiedades H(A)} guarantees that $A_n(\mathbb{K})$ is Noetherian, Zariski, and a PBW deformation of $G(A_n(\mathbb{K}))$, $H(A_n(\mathbb{K}))$ is a domain graded Noetherian and Artin-Schelter regular, $G(A_n(\mathbb{K}))$ is Noetherian and Artin-Schelter regular, and ${\rm Rees}(A_n(\mathbb{K}))\cong H(A_n(\mathbb{K}))$. 

\vspace{0.2cm}

From Proposition \ref{prop.rel homogeniz Grad} and Theorem
\ref{teo.propiedades H(A)}, we immediately get the following result.
\begin{corollary}\label{cor.grad asociat}
If $A$ is a $\sigma$-filtered skew PBW extension over a ring $R$
defined by $R=\mathbb{K}\langle t_1, \dots, t_m\rangle/\langle
r_1,\dots, r_s\rangle$ such that $\sigma_i$ is graded, then $G(A)$ is a graded skew PBW extension over $G(R)$ in $n$ variables
$y_1,\dots, y_n$ given by
\begin{align} \begin{split}
y_it_k &= \sigma_i(t_k)y_i + r_i,\\
y_jy_i & =c_{i,j}y_iy_j+r_{0_{j,i}} + r_{1_{j,i}}y_{1} + \cdots +
r_{n_{j,i}}y_{n},
\end{split}
\end{align}
where
\begin{equation}
\begin{split}
r_i  & = \left\{\begin{aligned}
  \begin{array}{ll}
    \delta_i(t_k), & \text{ if } \deg(\delta_i(t_k)) = 2,\\
    0, & \text{otherwise},
  \end{array}
\end{aligned}\right.\\
r_{0_{j,i}} & =\left\{\begin{aligned}
                        \begin{array}{ll}
                          r_{0_{j,i}}, &\  \text{ if } \deg(r_{0_{j,i}}) = 2, \\
                          0, &\  \text{otherwise},
                        \end{array}
                      \end{aligned}\right.\\
r_{l_{j,i}} & =\left\{\begin{aligned}
                            \begin{array}{ll}
                             r_{l_{j,i}} , & \text{ if } \deg(r_{l_{j,i}}) = 1, \\
                              0, & \text{otherwise},
                            \end{array}
                          \end{aligned}\right.
\end{split}
\end{equation}
for $1\leq l\leq n$, with $c_{i,j},r_{0_{j,i}},  r_{1_{j,i}}, \dots,
r_{n_{j,i}}$, $1\leq i,j\leq n$, the constants that define $A$ as in (\ref{eq.rep variables x}), and $t_k\in G(R)$ is the coset of $t_k$.
\end{corollary}
\begin{proposition}\label{prop.ex dim 2}
PBW deformations of  Artin-Schelter regular algebras of dimension
two are $\sigma$-filtered skew PBW extensions.
\end{proposition}
\begin{proof}
By \cite[Corollary 2.13]{Gaddis2016}, PBW deformations of Artin-Schelter regular algebras of dimension two are isomorphic to one of the following algebras: $\mathbb{K}\langle x,y\rangle/\langle xy-qyx\rangle$, $\mathbb{K}\langle x,y\rangle/\langle xy-qyx+1\rangle$, $\mathbb{K}\langle x,y\rangle/\langle yx-xy+y\rangle$, $\mathbb{K}\langle x,y\rangle/\langle yx-xy+y^2\rangle$, $\mathbb{K}\langle x,y\rangle/\langle yx-xy+y^2+1\rangle$, where $q\in\mathbb{K}\setminus \{0\}$. Notice that the first three algebras are skew PBW extensions over $\mathbb{K}$ and the last two are skew PBW extensions over $\mathbb{K}[y]$. As one can check, every algebra satisfies the conditions (1) and (2) established in Theorem \ref{teo.filtracion gen}.
\end{proof}
\begin{example}\label{rem-nueva}
\begin{enumerate}
    \item [\rm (1)] Andruskiewitsch, Dumas and Pe\~{n}a \cite{AndruskiewitschDumas2021} studied the
Hopf algebra $\mathcal{D}$ which was called  \emph{the double of the Jordan plane}. The authors considered the field $\mathbb{K}$ to be characteristic zero and algebraically closed. Following \cite[Definition 2.1]{AndruskiewitschDumas2021}, the Hopf algebra $\mathcal{D}$ is presented by generators $u, v, \zeta, g^{\pm 1}, x, y$ and relations $g^{\pm 1}g^{\pm 1} = 1$,\quad $\zeta g = g\zeta$,\quad $gx = xg$, \quad $gy = yg + xg$, \quad $\zeta y = y\zeta + y$, \quad $\zeta x = x\zeta + x$,  \quad $ug = gu$, \quad $vg = gv + gu$, \quad $v\zeta = \zeta v + v$, \quad $u\zeta = \zeta u + u$, \quad $yx = xy - \frac{1}{2}x^2$, \quad $vu = uv -\frac{1}{2}u^2$, \quad $ux = xu$, \quad $vx = xv + (1 - g) + xu$, $uy = yu + (1 - g)$, \quad $vy = yv - g\zeta + yu$. According to \cite[Lemma 4.1]{AndruskiewitschDumas2021}, the algebra $\mathcal{D}$ can be described as the iterated Ore extension
\[
\mathcal{D}\cong \mathbb{K}[g^{\pm 1}, x, u][y; \sigma_1, \delta_1][\zeta; \sigma_2, \delta_2][v;\sigma_3, \delta_3],
\]
where $\sigma_1$ the identity automorphism of $\mathbb{K}[g^{\pm 1}, x, u]$, and $\delta_1$ is the $\sigma_1$-derivation of $\mathbb{K}[g^{\pm 1}, x, u]$ given by $\delta_1(x)=-\frac{1}{2}x^2,\ \delta_1(u)=g-1$, and $\delta_1(g)=-xg$; $\sigma_2$ is the identity automorphism of $\mathbb{K}[g^{\pm 1}, x, u][y; \sigma_1, \delta_1]$, and $\delta_2$ is the $\sigma_2$-derivation of $\mathbb{K}[g^{\pm 1}, x, u][y; \sigma_1, \delta_1]$ defined by
$\delta_2(x)=x,\ \delta_2(u)=-u,\ \delta_2(g)=0$, and $\delta_2(y)=y$. Finally, $\sigma_3$ and $\delta_3$ are the automorphism and the $\sigma_3$-derivation of $\mathbb{K}[g^{\pm 1}, x, u][y; \sigma_1, \delta_1][\zeta; \sigma_2, \delta_2]$, respectively, given by $\sigma_3(x)=x,\ \sigma_3(u)=u,\ 
\sigma_3(g)=g,\ \sigma_3(y)=y$ and $\sigma_3(\zeta)=\zeta+1$, $\delta_3(x)=1-g+xu,\ \delta_3(u)=-\frac{1}{2}u^2,\ 
\delta_3(g)=gu,\ \delta_3(y)=yu-g\zeta$ and $\delta_3(\zeta)=0$. 

\vspace{0.2cm}

Notice that $\mathbb{K}[g^{\pm 1}, x, u][y;
\sigma_1, \delta_1][\zeta; \sigma_2, \delta_2][v;\sigma_3,
\delta_3]$ satisfies the four conditions established in
\cite[Example 2.2 of Part I]{Fajardoetal2020}, which means that $\mathcal{D}$ is a bijective skew PBW extension over $\mathbb{K}[g^{\pm 1}, x, u]$, that is, 
\[
\mathcal{D}\cong \mathbb{K}[g^{\pm 1}, x, u][y; \sigma_1,
\delta_1][\zeta; \sigma_2, \delta_2]\cong \sigma(\mathbb{K}[g^{\pm 1}, x, u])\langle
y,\zeta,v\rangle.
\]
Notice that $\sigma_i$ and $\delta_i$ restricted to
$\mathbb{K}[g^{\pm 1}, x, u]$ are the endomorphism and derivation as in Proposition
\ref{sigmadefinition}, whence $\mathcal{D}$ satisfies the conditions of Theorem \ref{teo.filtracion gen}, and therefore $\mathcal{D}$ is a $\sigma$-filtered skew PBW extension. 
\item [\rm (2)] Semi-graded rings were defined by Lezama and Latorre  \cite{LezamaLatorre2017} in the following way: a ring $R$ is called \emph{semi-graded} (SG) if there exists a collection $\{R_p\}_{p\in \mathbb{N}}$ of subgroups $R_p$ of the additive group $R^{+}$ such that the following conditions hold:
\begin{itemize}
    \item $R = \bigoplus_{p\in \mathbb{N}}R_p$;
    \item For every $p,q\in \mathbb{N}$, $R_p R_q\subseteq R_0\oplus R_1\oplus\cdots \oplus R_{p+q}$;
    \item $1\in R_0$.
\end{itemize}
Notice that $R$ has a standard $\mathbb{N}$-filtration given by $\mathcal{F}_p(R):= R_0 \oplus\cdots \oplus R_p$ (\cite[Proposition 2.6]{LezamaLatorre2017}), and $\mathbb{N}$-graded rings and skew PBW extensions are examples of semi-graded rings (\cite[Proposition 2.7]{LezamaLatorre2017}). In the case of a skew PBW extension $A$ over a ring $R$, they assumed $A_0=R$, i.e. $R$ has the trivial positive filtration. Notice that under these conditions, skew PBW extensions over an algebra $R$ with the standard $\mathbb{N}$-filtration are $\sigma$-filtered. In this way, if $R$ does not have the trivial positive filtration, then $A$ is not generally $\sigma$-filtered, as can be seen in Remark \ref{rem.no filter}. 

\vspace{0.2cm}

Recently, Lezama \cite[Definition 4.3]{Lezama2021} introduced the notion of \emph{semi-graded Artin-Schelter regular algebra}, and proved under certain assumptions (\cite[Theorem 4.14]{Lezama2021}) that skew PBW extensions are semi-graded Artin-Schelter regular. With this purpose, he showed that $A$ is a connected semi-graded algebra with semi-graduation $A_0=\mathbb{K}$, and $A_p$ is the $\mathbb{K}$-subspace generated by $R_qx^{\alpha}$
such that $q+ |\alpha| = p$, for $p \geq 1$. In this regard, notice that $A$ with the standard $\mathbb{N}$-filtration (\cite[Proposition 2.6]{LezamaLatorre2017}) given by this semi-graduation is $\sigma$-filtered. 
\item [\rm (3)] Zhang and Zhang \cite{Zhang} defined double Ore extensions as a generalization of Ore extensions. If $R$ is an algebra, and $B$ is another algebra containing $R$ as a subring, then $B$ is a \emph{right double Ore extension} of $R$ if the following conditions hold:
\begin{itemize}
\item $B$ is generated by $R$ and two new variables $x_1$ and $x_2$.
\item The variables $x_1$ and $x_2$ satisfy the relation
\begin{equation*}
x_2x_1 = p_{12}x_1x_2 + p_{11}x_1^2 + \tau_1x_1 + \tau_2x_2 +
\tau_0,
\end{equation*}
where $p_{12}, p_{11}\in \mathbb{K}$ and $\tau_1, \tau_2, \tau_0 \in R$. 
\item As a left $R$-module, $B =\sum\limits_{\alpha_1,\alpha_2\geq 0}Rx_1^{\alpha_1}x_2^{\alpha_2}$ and it is a left free $R$-module with basis the set $\{x_1^{\alpha_1}x_2^{\alpha_2} \mid\alpha_1\geq 0,\alpha_2\geq 0\}$.
\item  $x_1R + x_2R \subseteq Rx_1 + Rx_2 + R$.
\end{itemize}
Left double Ore extensions are defined similarly. $B$ is a \emph{double Ore extension} if it is  left and right double Ore extension of $R$ with the same generating set $\{x_1,x_2\}$ (\cite[Definition 1.3]{Zhang}). $B$ is a \emph{graded right} ({\em left}) {\em double Ore extension} if all relations of $B$ are homogeneous with assignment ${\rm deg}(x_1) = {\rm deg}(x_2) =1$. They studied the property of being Artin-Schelter for these extensions (\cite{Zhang}, Theorem 3.3). 

\vspace{0.2cm}

Later, in \cite{Zhang2}, the same authors constructed 26 families of Artin-Schelter regular algebras of global dimension four using double Ore extensions. Briefly, to prove that a connected graded double Ore extension of an Artin-Schelter regular algebra is Artin-Schelter regular, Zhang and Zhang needed to pass the Artin-Schelter regularity from the trimmed double
extension $R_P[x_1, x_2; \sigma]$ to $R_P[x_1,x_2; \sigma,\delta,\tau]$ (details about the notation used for double Ore extensions can be seen in \cite{Zhang}). For this purpose they defined a new grading and with this a filtration: let $A=R_P[x_1,x_2; \sigma,\delta,\tau]$ be a graded (or ungraded) double extension of $R$ with $d_1 = \deg(x_1)$ and $d_2 = \deg(x_2)$ (or $\deg(x_1) = \deg(x_2)= 0$), the new defined graduation is $\deg'(x_1) = d_1 +1$ and $\deg'(x_2) = d_2 +1$ and $\deg'(r) = \deg(r)$ for all $r\in R$. Using this grading they defined a filtration of $A$ by $\mathcal{F}_p(A) = \{\sum r_{n_1,n_2}x_1^{n_1}x_2^{n_2}\in A\mid \deg'(r_{n_1,n_2} + n_1 \deg'(x_1) + n_2\deg'(x_2)\leq  m\}$. $\mathcal{F}=\{\mathcal{F}_p(A)\}_{p\in\mathbb{Z}}$ is an $\mathbb{N}$-filtration such that the associated graded ring ${G_{\mathcal{F}}}(A)$ is isomorphic to $R_P[x_1, x_2; \sigma]$; there is a central element $t$ of degree 1 such that ${\rm Rees}_{\mathcal{F}}(A)/(t) = R_P[x_1, x_2;\sigma]$ as graded rings and ${\rm Rees}_{\mathcal{F}}(A)/(t - 1) = R_P[x_1,x_2; \sigma,\delta,\tau]$ as ungraded rings; also, if $A$ is connected graded, then so are ${G_{\mathcal{F}}}(A)$ and ${\rm Rees}_{\mathcal{F}}(A)$, where ${\rm Rees}_{\mathcal{F}}(A)$ is the Rees ring associated to this filtration (\cite[Lemma 3.4]{Zhang}).

\vspace{0.2cm}

Related with this work, G\'omez and the first author proved that for $R= \bigoplus_{m\geq 0}R_m$ be an $\mathbb{N}$-graded algebra and $A=R_P[x_1,x_2; \sigma,\delta,\tau]$ be a graded right double Ore extension of $R$, if $P=\{p_{12},0\}$, $p_{12}\neq 0$ and $\sigma := \left(\begin{smallmatrix} \sigma_{11} && 0 \\ 0 && \sigma_{22} \\ \end{smallmatrix}\right)$, where $\sigma_{11}$, $\sigma_{22}$ are automorphism of $R$, then $A$ is a graded skew PBW extension over $R$ (\cite[Theorem 3.5]{GomezSuarez2019}). As one can check, the previous filtration on $A$ coincides with the filtration given in (\ref{eq.filtr A}), and so $A=R_P[x_1,x_2; \sigma,\delta,\tau]$ is a $\sigma$-filtered skew PBW extension.
\end{enumerate}
\end{example}

For the last theorem of the paper, recall that the \emph{enveloping algebra} of an algebra $R$ is the tensor product $R^e = R\otimes R^{op}$, where $R^{op}$ is the opposite algebra of $R$. If $M$ is an $R$-bimodule, and $\nu$, $\mu : R\to R$ are two automorphisms, then the skew $R$-bimodule $^\nu M^ \mu$ is equal to $M$ as a vector space with $a\cdot m\cdot b:=\nu(a)\cdot m\cdot \mu(b)$. When $\nu$ is the identity, we omit it. $M$ is a left $R^e$-module with product given by $(a\otimes b)\cdot m=a\cdot m\cdot b=\nu(a)\cdot m\cdot \mu(b)$. In particular, for $R$ and
$R^e$, we have the structure of left $R^e$-module given by
$(a\otimes b)\cdot x=\nu(a)x\mu(b)$, $(a\otimes b)\cdot (x\otimes
y)=a\cdot (x\otimes y)\cdot b=\nu(a)\cdot (x\otimes y)\cdot
\mu(b)=\nu(a)x\otimes y\mu(b)$.

\vspace{0.2cm}

An algebra $R$ is said to be \emph{skew Calabi-Yau} of dimension $d$ if it has a finite resolution by finitely generated projective bimodules, and there exists an algebra automorphism $\nu$ of $R$ such that
\[
{\rm Ext}^i_{R^e} (R,R^e) \cong \left\{
                \begin{array}{ll}
                                 0, & i\neq d, \\
                                 R^{\nu}, & i= d.
                                 \end{array}
                \right.
                \]
 as $R^e$ -modules. If $\nu$ is the identity, then $R$ is said to be \emph{Calabi-Yau}. Enveloping algebras and skew Calabi-Yau algebras related to skew PBW extensions were studied in \cite{ReyesSuarez2017}.
\begin{theorem}\label{teo.skew CY} If $A$ is a $\sigma$-filtered skew PBW extension over a ring $R$ such that $H(R)$ is Auslander-regular, then $A$ is skew Calabi-Yau.
\end{theorem}
\begin{proof}
From Theorem \ref{teo.propiedades H(A)}, parts (4), (5) and (9), we know that $A$ is a PBW deformation of a Noetherian Artin-Schelter regular algebra
$G(A)$. The assertion follows from \cite[Proposition 2.15]{Gaddis2016}.
\end{proof}
\begin{example} Let $R=\mathbb{K}[t_1,\dots, t_m]$, $m\geq 0$. Since $H(R)=R[z]$ is Auslander-regular and a commutative polynomial ring in $m+1$ variables over $\mathbb{K}$, then every one of the examples of skew PBW extensions over $R$ presented in \cite{Fajardoetal2020}, \cite{GallegoLezama}, \cite{LezamaReyes}, \cite{Suarez}, and 
\cite{Suarezthesis} that are $\sigma$-filtered, are skew Calabi-Yau. In particular, every one of the algebras in the proof of Proposition \ref{prop.ex dim 2} are skew Calabi-Yau.
\end{example}
\section{Future work}\label{confutwork}
Hausdorf, Seiler and Steinwandt \cite{Hausdorfetal2002} gave solution to the problem of the completion of the normal form algorithm for Gr\"obner bases in Weyl algebras using the technique of homogenization, with the aim to define Gr\"obner bases and (weak) involutive bases for non-term orders. Since Weyl algebras are skew PBW extensions, and Gr\"obner basis theory of these objects was formulated in \cite{LezamaGallego2017}, we can think of trying to generalize its theory to the setting of $\sigma$-filtered skew PBW extensions using the results established in this paper, and hence, to formulate a theory of involutive bases for skew PBW extensions.

\vspace{0.2cm}

On the other hand, since Redman \cite{Redman1999} and Chirvasitu et al. \cite{Chirvasitu2018} studied the noncommutative geometry of the homogenization of two classes of three dimensional skew polynomial algebras, and of the homogenization of the universal enveloping algebra $U(\mathfrak{sl}(2, \mathbb{C}))$, respectively, keeping in mind that these algebras are particular examples of skew PBW extensions (Examples \ref{ex-tresdimensional} and \ref{ex-Liealgebratresdimensional}), we can think of establishing several properties of  noncommutative geometry of the homogenization of skew PBW extensions. It is a natural task to investigate if the treatment developed by Redman and Chirvasitu et al. can be extended to the more general setting of these extensions.


\end{document}